\documentclass[12pt]{amsart}

\usepackage{geometry}                
\usepackage{graphicx}
\usepackage{amsmath}
\usepackage{amssymb}
\usepackage{epstopdf}
\usepackage{enumerate}
\newtheorem{thm}{Theorem}[section]
\newtheorem{lemma}[thm]{Lemma}

\newtheorem{prop}[thm]{Proposition}

\newcommand{\N}{\mathbb{N}}

\begin{document}
	
\title[Self-affine sets with holes]{Dimension of self-affine sets with holes}

\author{Andrew Ferguson}
\address{Andrew Ferguson\\School of Mathematics\\University of Bristol\\ University Walk\\Bristol\\BS8 1TW\\UK.}
\email{andrew.ferguson@bris.ac.uk}

\author{Thomas Jordan}
\address{Thomas Jordan\\School of Mathematics\\University of Bristol
\\University Walk\\Bristol\\BS8 1TW\\UK.}
\email{thomas.jordan@bris.ac.uk}

\author{Micha\l \,Rams}
\address{Micha\l\, Rams\\Institute of Mathematics\\Polish Academy of Sciences\\
UL. \'{S}niadeckich 8\\ 00-956 Warszawa\\ Poland.}
\email{m.rams@impan.gov.pl}

\thanks{A.F. acknowledges support from EPSRC grant EP/I024328/1 and the University of Bristol. M.R. was partially supported by MNiSWN grant N201 607640 (Poland).}

\date{\today}                                       

\maketitle

\begin{abstract}In this paper we compute the dimension of a class of dynamically defined non-conformal sets.  Let $X\subseteq\mathbb{T}^2$ denote a Bedford-McMullen set and $T:X\to X$ the natural expanding toral endomorphism which leaves $X$ invariant.  For an open set $U\subset X$ we let  \begin{equation*}X_U=\{x\in X\,:\,T^k(x)\not\in U\text{ for all }k\}.\end{equation*}  We investigate the box and Hausdorff dimensions of $X_U$ for both a fixed Markov hole and also when $U$ is a shrinking metric ball.  We show that the box dimension is controlled by the escape rate of the measure of maximal entropy through $U$, while the Hausdorff dimension depends on the escape rate of the measure of maximal dimension. 
\end{abstract}

\section{Introduction and statement of results}

Let $(X,d)$ be a compact metric space and $T:X\to X$ a continuous map.  For an open set $U\subset X$ we consider the set of points which under forward iteration do not enter $U$, i.e.

\begin{equation*}X_U=\{x\in X\,:\,T^k(x)\not\in U\text{ for all }k\}.\end{equation*} The purpose of this paper is to investigate the box and Hausdorff dimensions of the set $X_U$.  This problem is far from new.  Urba\'{n}ski considered precisely this question for uniformly expanding \cite{Urb87} and non-uniformly expanding \cite{Urb89} endomorphisms of the circle proving, amongst other things, that the map $\epsilon\to {\rm dim}_H(X_{B_{\epsilon}(z)})$ is continuous.
	
In \cite{Hen92} Hensley considered this problem in the setting of continued fractions: for $x\in (0,1)\setminus \mathbb{Q}$ we write 
\begin{equation*}x=[a_1,a_2,\cdots]=\frac{1}{a_1+\frac{1}{a_2+\cdots}}.\end{equation*} For a positive integer $n$ we let 
\begin{equation*}E_n=\{x=[a_1,a_2,\ldots]\,:\,a_k\leq n\text{ for all }k\}.\end{equation*}  

The set $E_n$ may be understood in terms of the Gauss map $T:[0,1]\to [0,1]$ defined by 
\begin{equation*}T(x)=\begin{cases} \{\frac{1}{x}\} & \text{ if }x\neq 0 \\
	0 & \text{ if }x=0.\end{cases}\end{equation*} 
Then the set $E_n$ is the set of points which do not enter the set $[0,(n+1)^{-1})$ under forward iteration.	
	
Hensley obtained quite detailed bounds on the spectral radius of a perturbed transfer operator which was then used to prove an asymptotic formula for the dimension of $E_N$, 

\begin{equation*}{\rm dim}_H(E_n)=1-\frac{6}{\pi^2 n}-\frac{72\log n}{\pi^4 n^2}+O(n^{-2}).\end{equation*}

Another result of interest in this line of enquiry is that of Liverani and Maume-Deschamps \cite{LivMau03} who proved that if $T$ is a Lasota-Yorke map then the Hausdorff dimension of $X_U$ map be computed implicitly using a formula that is analogous to Bowen's equation.  

A central object in the study of dynamical systems with holes is the \emph{escape rate} of a measure.  If $\mu$ is a $T$-invariant probability measure we define the escape rate of $\mu$ through $U$ to be the quantity
	
\begin{equation}\label{eq:escratdef}r_\mu(B_\epsilon(z))=-\limsup_{k\to\infty} \frac{1}{k}\log\mu\{x\in X\,:\,T^i(x)\not\in B_\epsilon(z)\text{ for }0\leq i<k \}\end{equation}
	
Several of our results rely heavily on recent advances in the dependence of the quantity $r_\mu(B_\epsilon(z))$ has on $\epsilon$ and $z$.  We now give a brief overview of these developments.  

In \cite{BunYur08} Bunimovich and Yurchenko considered the case that $T$ is the doubling map and $\mu$ the Lebesgue measure, proving that

\begin{equation}\lim_{n\to\infty}\frac{r_{\mu}(I_n(z))}{\mu(I_n(z))}=\begin{cases} 1 & \text{ if }z\text{ is non-periodic} \\
1-2^{-p} & \text{ if }z\text{ has prime period }p\end{cases}\label{eq:BunYur}	\end{equation} here $\{I_n(z)\}_{n=1}^{\infty}$ denotes a nested family of dyadic intervals for which $\bigcap_{n=1}^{\infty} I_n(z)=\{z\}$. 

Later, Keller and Liverani \cite{KelLiv09} proved a general perturbation result which, providing the correct functional analytic setup holds, yields a first order expansion for the spectral radius of perturbed transfer operator.  The leading term in this expansion displays a similar dependence on how preimages of $B_\epsilon(z)$ intersect, which in the case of a uniformly expanding map reduces to the periodicity of $z$.  This perturbation result was then applied to the setting of piecewise expanding maps of the interval to obtain various statistical results, including a generalisation of the equation (\ref{eq:BunYur}).  A further refinement to this formula, which includes both smooth and non-smooth higher order terms, for the case of the doubling map was recently obtained by Dettmann \cite{Det11}.

In \cite{FerPol10} Pollicott and the first author show that this functional setup  \cite{KelLiv99,KelLiv09} applies in the setting of subshifts of finite type and then use an approximation argument to show  that similar conclusions can be arrived at when $T$ is conformal and expanding  and $\mu$ a Gibbs measure.  Another problem considered in that paper was the behaviour of the Hausdorff dimension of $X_{B_\epsilon(z)}$ for $\epsilon$ small.  Suppose that $T$ is $C^{1+\alpha}$ and expanding and that $X$ is a repeller for $T$.  By a result of Ruelle \cite{Rue82} the Hausdorff dimension of $X$ is given implicitly by Bowen's equation, that is $s={\rm dim}_H(X)$ where 
\begin{equation*} P(-s\log | dT |):=\sup\left\{h_\nu - s \int \log |d T | d\nu\,:\,\nu \text{ is }T\text{-invariant and }\nu(X)=1\right\}=0.\end{equation*} 
Providing the map $T$ is topologically mixing this supremum is attained by a unique $T$-ergodic measure, equivalent to the $s$-dimensional Hausdorff measure, which we denote by $\mu$.  Under these assumptions it was shown \cite{FerPol10}[Theorem 1.2] that
\begin{equation*} \lim_{\epsilon\to 0} \frac{s-{\rm dim}_H(X_{B_\epsilon(z)})}{\mu(B_\epsilon(z))} = \frac{1}{\int \log | dT| d\mu }\begin{cases}  1 & \text{ if }z\text{ is non-periodic} \\
	1- |d_z T^p |^{-s} & \text{ if }z\text{ has prime period }p\end{cases}\end{equation*} 
 
In this paper we continue this line of work by investigating the box and Hausdorff dimensions of $X_U$ for a class of non-conformal systems known as Bedford-McMullen sets, which we now briefly describe: Fix integers $2\leq m < n$ and let $D\subset\{0,1,\ldots, n-1\}\times \{0,1,\ldots,m-1\}$.  For $(i,j)\in D$ write

\begin{equation*}F_{(i,j)}(x,y)=\left(\frac{x+i}{n},\frac{y+j}{m}\right).\end{equation*}

Let $X$ denote the unique non-empty compact set satisfying

\begin{equation*}X=\bigcup_{(i,j)\in D}F_{(i,j)}(X).\end{equation*}

	\begin{figure}
	\begin{center}
	\includegraphics[width=0.8\textwidth]{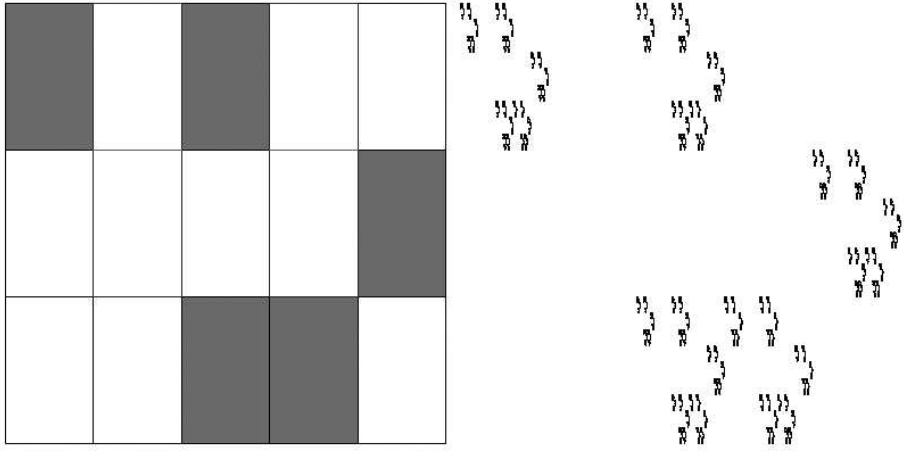}
	\end{center}
	\caption{A generating pattern of a Bedford-McMullen carpet (left) and the associated invariant set $X$ (right).}\label{fig:pic}
	\end{figure}

Sets of this kind were first studied by Bedford \cite{Bed84} and McMullen \cite{McM84} who independently calculated the Hausdorff and box dimension of $X$. They obtained that if $z(j)=\#\{0\leq i < n \,:\, (i,j)\in D\}$, $\eta=\frac{\log m}{\log n}$ and $\pi_D:D\to\{0,1,\ldots,m-1\}$ denotes projection onto the second coordinate then

\begin{equation}\label{eq:hausdef}{\rm dim}_H(X)=s:=\frac{1}{\log m}\log\left(\sum_{ j=0}^{m-1} z(j)^\eta \right)\end{equation} 
and  
\begin{equation}\label{eq:boxdef}{\rm dim}_B(X)=\frac{\log \# D}{\log n}+(1-\eta)\frac{\log \#\pi_D(D)}{\log m}.\end{equation}  
	For extensions of these formulae to sets modelled by subshifts we refer the reader to \cite{KenPer96, KenPer96a, Oli10, Yay09}.
	
When considered as a subset of the torus $\mathbb{T}^2$, the set $X$ is invariant under the expanding toral endomorphism $T(x,y)=(nx,my) \mod 1$.  
	
Let $U\subset \mathbb{T}^2$ consist of a finite union of sets of the form $[i n^{-l},(i+1)n^{-l}]\times [j m^{-k},(j+1)m^{-k}]$ for positive integers $i,j,l,k$, we shall refer to such a set as \emph{Markov}.  Our first two results concern the dependence of the box and Hausdorff dimensions of $X_U$ on such a Markov hole.

Let $\mu_{{\rm max}}$ denote the measure of maximal entropy for $T:X\to X$. i.e. the unique measure satisfying

\begin{equation*} \mu_{{\rm max}} = (\# D)^{-1}\sum_{(i,j)\in D}  (F_{i,j})_*(\mu_{{\rm max}}). \end{equation*} We remark that this corresponds to the $(1 /\#D)_{(i,j)\in D}$ Bernoulli measure on the full shift $D^\mathbb{N}$. However, in general, this is not the measure of maximal Hausdorff dimension. The measure of maximal dimension is the self-affine measure with weights $(\frac{z(j)^{\eta-1}}{m^s})_{(i,j)\in D}$, where $s=\dim_{H} X$. In other words the unique measure satisfying
\begin{equation*} \mu_{{\rm dim}} = \sum_{(i,j)\in D} \frac{z(j)^{\eta-1}}{m^s} (F_{i,j})_*(\mu_{{\rm dim}}). \end{equation*} 

  Let $\pi:\mathbb{T}^2\to\mathbb{S}^1$ denote the projection onto the second coordinate.  If $S:\mathbb{S}^1\to\mathbb{S}^1$ denotes multiplication by $m \mod 1$ then it is easy to see that $\pi T=S\pi$.  Let $\tilde{\mu}_{{\rm max}}$ denote the measure of maximal entropy for $S:\pi(X)\to\pi(X)$.  The box dimension of $X_U$ depends not only on the hole $U$ but also on the `size' of $\pi(X_U)$: let

\begin{equation*}\tilde{U}=\{y\in\pi(U)\,:\,\pi^{-1}\{y\}\subset U\}\end{equation*} in which case we see that $\pi(X_U)=\pi(X)\setminus \bigcup_{k=0}^{\infty}S^{-k}(\tilde{U})$.
	
We now state the result concerning the box dimension of the survivor set $X_U$.

\begin{thm}Suppose that $U\subset X$ is a Markov sets and that $T:X_U\to X_U$ is topologically mixing, then
\begin{equation}\label{eq:LDbox}{\rm dim}_B(X_U)={\rm dim}_B(X)-\frac{\eta r_{\mu_{\rm max}}(U)+(1-\eta)r_{\tilde{\mu}_{\rm max}}(\tilde{U})}{\log m}\end{equation}	 where $r_{\cdot}(\cdot)$ denotes the escape rate as defined in equation (\ref{eq:escratdef}).
	 \label{thm:boxMark}\end{thm}

For a probability vector $\underline{p}=(p_d)_{d\in D}\in \Delta_D=\{(q_d)_{d\in D}\:\,q_d\geq 0\text{ for all }d\in D\text{ and }\sum_{q\in D}q_d=1\}$ we let $\mu_{\underline{p}}$ denote the associated Bernoulli measure, that is the 	unique Borel probability measure satisfying

\begin{equation*} \mu_{\underline{p}} = \sum_{(i,j)\in D} p_{(i,j)} (F_{i,j})_*(\mu_{\underline{p}}). \end{equation*}
	
Concerning the Hausdorff dimension of $X_U$ we obtain the following bounds.

\begin{thm}Suppose that $U\subset X$ is a Markov sets and that $T:X_U\to X_U$ is topologically mixing, then

\begin{equation}\label{eq:LDhaus}{\rm dim}_H(X_U)\leq \sup_{\underline{p}\in\Delta_D}\left\{{\rm dim}_H(\mu_{\underline{p}})-\frac{\eta r_{\mu_{\underline{p}}}(U)+(1-\eta)r_{\pi_*(\underline{p})}(\tilde{U})}{\log m}\right\}.\end{equation}
Furthermore, we obtain lower bounds for the two extreme cases $\tilde{U}=\emptyset$ and $\tilde{U}=\pi(U)$.  Let $l\in\mathbb{N}$ denote any integer for which $[i n^{-l},(i+1)n^{-l}]\times [j m^{-l},(j+1)m^{-l}]\cap U\neq \emptyset$ implies that $[i n^{-l},(i+1)n^{-l}]\times [j m^{-l},(j+1)m^{-l}]\subseteq U$.  In which case:
\begin{equation*}{\rm dim}_H(X_U) \geq \begin{cases} \sup_{\underline{p}\in\Delta_D}\left\{{\rm dim}_H(\mu_{\underline{p}})+\frac{1}{\log n }\int \log\left(\frac{\mu_{\underline{p}}(\pi^{-1}(I_l(y))\cap U^c)}{\pi_*(\mu_{\underline{p}})(I_l(y))}\right) d\pi_*(\mu_{\underline{p}})(y)\right\} & \text{ if }\tilde{U}=\emptyset\\
	\sup_{\underline{p}\in\Delta_D}\left\{{\rm dim}_H(\mu_{\underline{p}})-\frac{r_{\mu_{\underline{p}}}(U)}{\log m}\right\} & \text{ if }\tilde{U}=\pi(U)\end{cases}\end{equation*} where for $I_l(y)$ denotes the unique interval of the form $[j m^{-l},(j+1)m^{-l})$ containing  $y$.
\label{thm:hausMark}\end{thm}

Using an approximation argument we are able to extend Theorems \ref{thm:boxMark} and \ref{thm:hausMark} to the case of shrinking metric balls. For a measure $\nu$ with support in $X$ and $x\in {\rm supp}(\nu)$ we define the lower local dimension of $\nu$ at $x$ to be the quantity

\begin{equation*}\underline{{\rm dim}}_{\rm loc}(\nu)(x)=\liminf_{\epsilon\to 0}\frac{\log\nu(B_\epsilon(x))}{\log \epsilon}.\end{equation*}  

As with in the case of escape rates and the dimension for conformal systems with holes we find that the position of the hole has an effect on the dimension.  Define functions $d_B,\tilde{d}_B:X\to\mathbb{R}$  by

\begin{equation*}\begin{split}d_B(z)= & \begin{cases} 1 & \text{ if }z\text{ is non-periodic} \\
	1-\#D^{-p} & \text{ if }z\text{ has prime period }p.\end{cases}\\
	\tilde{d}_B(z)= & \begin{cases} 1 & \text{ if }\pi(z)\text{ is non-periodic} \\
		1-\#\pi(D)^{-p} & \text{ if }\pi(z)\text{ has prime period }p.\end{cases}\end{split}\end{equation*}
		
Our first result regarding the behaviour of the box dimensions of the set $X_{B_\epsilon(z)}$ is the following.

\begin{thm}Let $\mu_{{\rm max}}$ denote the measure of maximal entropy for $T:X\to X$.  Let $\nu$ denote a $T$-invariant ergodic Borel probability measure with ${\rm supp}(\nu)=X$.  Furthermore we assume that if $(0,0)\in X$ then $\underline{{\rm dim}}_{\rm loc}(\pi_*(\nu))(0)>0$. We have
\begin{enumerate}
	\item If $ z(j) \leq 1$ for all $j$ then for $\nu$-almost all $z$
	\begin{equation*}\lim_{\epsilon\to 0}\frac{{\rm dim}_B(X)-\overline{{\rm dim}}_B(X_{B_\epsilon(z)})}{\mu_{\rm max}(B_\epsilon(z))} = \frac{1}{\log m}\left(\eta d_B(z)+(1-\eta)\tilde{d}_B(z)\right).\end{equation*}
		\item If $z(j) > 1$ for some $j$ then for $\nu$-almost all $z$
		\begin{equation*} \lim_{\epsilon\to 0} \frac{{\rm dim}_B(X)-\overline{{\rm dim}}_B(X_{B_\epsilon(z)})}{\mu_{\rm max}(B_\epsilon(z))} = \frac{\eta d_B(z)}{\log m}.\end{equation*}
\end{enumerate}

The same holds also for the lower box dimension $\underline{{\rm dim}}_B$.  
\label{thm:boxmet}
\end{thm}

For Hausdorff dimension we obtain a similar result under the additional assumption that the measures $\nu$ and $\pi_*(\nu)$ are non-atomic.  This allows us to disregard the case that $z$ (or $\pi(z)$) is periodic which is a function of the less than optimal bounds derived in Theorem \ref{thm:hausMark}. 

\begin{thm}Let $\mu_{{\rm dim}}$ denote the measure of maximal dimension for $T:X\to X$.  Let $\nu$ denote a $T$-invariant Borel probability measure with ${\rm supp}(\nu)=X$.  Assume that $\pi_*(\nu)$ is non-atomic and that if $(0,0)\in X$ then $\underline{{\rm dim}}_{\rm loc}(\pi_*(\nu))(0)>0$.  
	\begin{enumerate}
		\item If $ z(j) \leq 1$ for all $j$ we have that for $\nu$-almost all $z$
		\begin{equation*}\lim_{\epsilon\to 0}\frac{{\rm dim}_H(X)-{\rm dim}_H(X_{B_\epsilon(z)})}{\mu_{\rm dim}(B_\epsilon(z))} = \frac{1}{\log m}.\end{equation*}
			\item If $z(j) > 1$ for some $j$ then we have that for-$\nu$ almost all $z$
			\begin{equation*} \lim_{\epsilon\to 0} \frac{{\rm dim}_H(X)-{\rm dim}_H(X_{B_\epsilon(z)})}{\mu_{\rm dim}(B_\epsilon(z))} = \frac{1}{\log n}.\end{equation*}
	\end{enumerate}
\label{thm:hausmet}
\end{thm}

The remainder of this paper is structured as follows: section \ref{sec:prelim} contains a description of some of the tools that we make use of in proving these results.  Sections \ref{sec:boxMark} and \ref{sec:hausMark} contain a proofs of Theorems \ref{thm:boxMark} and \ref{thm:hausMark} respectively.  In section \ref{sec:approx} we prove two approximation results that allow us to approximate metric balls with Markov squares.  Finally, in section \ref{sec:metric} we combine sections \ref{sec:boxMark},\ref{sec:hausMark} and \ref{sec:approx} to deduce Theorems \ref{thm:boxmet} and \ref{thm:hausmet}.

\section{Preliminaries}
\label{sec:prelim}

\subsection{Perturbations of the transfer operator}

Central to the proof of Theorems \ref{thm:boxMark} and \ref{thm:hausMark} is the estimation of the the number of strings in the symbolic space which avoid a particular collection of forbidden words.  For this we make use of the perturbation theory of transfer operators \cite{ColMarSch97,KelLiv99,KelLiv09}.  

Let $A$ denote an irreducible and aperiodic $l\times l$ matrix of zeroes and ones, i.e. there exists a positive integer $d$ such that $A^{d}>0$.  We define the subshift of finite type (associated with matrix $A$) to be 

\begin{equation}\nonumber\Sigma=\{(x_n)_{n=0}^\infty\,:\,A(x_n,x_{n+1})=1,\text{ for all }n\}.\end{equation}  
	
If we equip the set $\{0,1,\ldots,l-1\}$ with the discrete topology then $\Sigma$ is compact in the corresponding Tychonov product topology.  The shift $\sigma:\Sigma\to\Sigma$ is defined by $\sigma(x)=y$, where $y_n=x_{n+1}$ for all $n$, i.e. the sequence is shifted one place to the left and the first entry deleted.  

For $\theta\in (0,1)$ we define a metric on $\Sigma$ by $d_\theta(x,y)=\theta^{m}$, where $m$ is the least non-negative integer (assuming that such a $m$ exists) with $x_m\neq y_m$, otherwise we set $d_\theta(x,x)=0$.  Equipped with the metric $d_\theta$, the space $(\Sigma,d_\theta)$ is complete, and moreover the topology induced by $d_\theta$ agrees with the previously mentioned Tychonov product topology.  Finally, for $x\in \Sigma$ and a positive integer $n\geq 1$ we define the cylinder of length $n$ centred on $x$ to be the set $[x]_n=[x_0,x_1,\ldots,x_{n-1}]=\{y \in\Sigma\,:\,y_i=x_i\text{ for }i=0,1,\ldots,n-1\}$.

Fix a $d_\theta$-Lipschitz continuous function $\phi:\Sigma\to\mathbb{R}$, and recall that we let $\mu$ denote its equilibrium state, i.e.,

\begin{equation}\nonumber P(\phi):=\sup\left\{h_{\nu}+\int \phi d\nu\,:\,\sigma_*(\nu)=\nu,\,\nu(\Sigma)=1\right\}=h_{\mu}+\int \phi d\mu.\end{equation}

We define the transfer operator $\mathcal{L}:C(\Sigma)\to C(\Sigma)$ acting on the space of continuous functions equipped with the supremum norm $\|\cdot\|_\infty$.

Writing $i=(i_0,i_1,\ldots,i_{k-1})$ for an allowed string of length $k$ then we may write $(\mathcal{L}^k w)(x)=\sum_{|i|=k}e^{\phi^k(ix)}w(ix)$ where the sum is over the strings for which the concatenation $ix$ is allowed, i.e. we require $ix\in\Sigma$.

Given a set $U\subset \Sigma$ consisting of a finite collection of cylinder sets we wish to study how the number of words of length $k$ that do not intersect grows $U$ as $k$ gets large.  Clearly this information is encoded by the escape rate of the measure of maximal entropy through $U$.  We state the following consequence of a result of Collet, Martinez and Schmidt \cite{ColMarSch97}.  

\begin{prop}Let $\mu$ denote the equilibrium state associated with a $d_\theta$-Lipschitz potential.  Denote by $\Sigma_U=\Sigma\setminus\bigcup_{k\geq 0}\sigma^{-k}(U)$ and suppose that $\sigma:\Sigma_U\to \Sigma_U$ is topologically mixing.  Then the following limit exists 
	\begin{equation*}r_\mu(U)=-\lim_{k\to\infty}\frac{1}{k}\log\mu\{x\in\Sigma\,:\,\sigma^i(x)\not\in U\text{ for }i<k\}\end{equation*} \label{thm:colmarshm}\end{prop}

Due to the non-conformality of the system $T:X\to X$ for the purposes of constructing a good cover to compute the dimensions of $X$ we are required to consider the dynamics of factor system $S:\pi(X)\to\pi(X)$.  If $\mathcal{U}$ is a cover of $X$ by approximate squares of side length $\approx m^{-k}$ then Proposition \ref{thm:colmarshm} may be used to yield information about the number of elements from this cover which do not hit some forbidden region $U$ (e.g. metric ball or finite union of Markov holes) for the first $[\eta k]$ iterates under $T$.  Likewise, by considering the dynamics of the factor $S:\pi(X)\to\pi(X)$ we may estimate the proportion of this cover that do not intersect $U$ for the iterates $[\eta k]< i \leq k$.  These two events (not hitting $U$ at times $\leq [\eta k]$ and not hitting $U$ at times $>[\eta k]$) are not independent and the following lemma provides a means of `gluing' these two estimates together and so estimating the number of approximate squares lost by forbidding intersection with $U$ under iteration. 

In order to this we introduce the following perturbation of the transfer operator $\mathcal{L}_{U}:C(\Sigma)\to C(\Sigma)$ given by $(\mathcal{L}_U w)(x)=(\mathcal{L} \chi_{U^c} w)(x)$.  

\begin{lemma}\label{thm:glue}Let $\mu$ denote the equilibrium state associated with a Lipshitz potential.  Let $U\subseteq\Sigma$ consist of a finite union of cylinders with the property that $\sigma:\Sigma_U\to\Sigma_U$ is topologically mixing.  There exists a constant $c=c(U)$ such that for all $x\in\Sigma_{U}$ and $l\in\mathbb{N}$ we have 
\begin{equation*}\mu\{y\in [x]_{l}\,:\,\sigma^i(y)\not\in U\text{ for }0\leq i <k\} \geq e^{-k r_\mu(U)}\left(c+o(1)\right).\end{equation*}\end{lemma}  

	\begin{proof}  Fix $x\in \Sigma_U$ then we observe that 
		\begin{equation*}\begin{split}  \mu\{y\in [x]_{l}\,:\,\sigma^i(y)\not\in U\text{ for }0\leq i <k\} & = \int_{[x]_{l}}\prod_{i=0}^{k-1}\chi_{U^c}(\sigma^i(y) ) d\mu(y) \\
			& = \int \left(\mathcal{L}^k\prod_{i=0}^{k-1}\chi_{U^c}\circ\sigma^i \chi_{[x]_l}\right)(y) d\mu(y) \\
			& = \int (\mathcal{L}^k_U \chi_{[x]_{l}})(y) d\mu(y) \\
			& \leq e^{-kr_\mu(U)}\left( \underbrace{\mu_U([x]_{l})\int g_U d\mu_U}_{:=A} \right. \\
			 &+\left. \underbrace{\left \| e^{k r_\mu(U)}(\mathcal{L}_U^k \chi_{[x]_{l}})-\mu_U([x]_{l})\int g_U d\mu_U \right\|_\infty}_{:=B} \right).\end{split}\end{equation*}

In \cite{ColMarSch97} it is shown that $B\to 0$ as $k\to \infty$.  Further, the estimates $h_U>0$ and $\mu_U([x]_{l})>0$ for any $x\in\Sigma_U$ are proven.  Taking the infimum over $x\in\Sigma_U$ in $A$ completes the proof.
\end{proof} 

\subsection{Shrinking holes}

The proofs of Theorems \ref{thm:boxmet} and \ref{thm:hausmet} make use of estimates on how the quantity $r_\mu(U)$ varies as $U$ shrinks to a point.  In order to do this we employ the framework of Keller and Liverani \cite{KelLiv99, KelLiv09} along with the estimates found in \cite{FerPol10} that show that this framework applies in the setting of subshifts of finite type.

The conditions that we impose on our sequence of holes $U_N\subset \Sigma$ are: 

	\begin{enumerate}[(a)]
		\item $\{U_N\}_N$ are nested with $\cap_{N\geq 1} U_N = \{z\}$.\label{ass1} 	
		\item Each $U_N$ consists of a finite union of cylinder sets, with each cylinder having length $l(N)$. \label{ass2}
		\item There exists a sequence $\{\rho_N\}_N\subset \mathbb{N}$, constants $\kappa>0$, $q\in\N$ and points $z^{(1)},z^{(2)},\ldots,z^{(q)}$ such that $\kappa<\rho_N/l(N)\leq 1$ and  $U_N\subset \cup_i [z^{(i)}]_{\rho_N}$ for all $N\geq 1$. \label{ass4}
		\item If $z$ is periodic with prime period $p$ then $\sigma^{-p}(U_N)\cap [z_0 z_1\cdots z_{p-1}]\subseteq U_N$ for large enough $N$. \label{ass5}
	\end{enumerate}

Under these assumptions one may conclude the following.

\begin{prop}Suppose that $\sigma:\Sigma\to\Sigma$ is topologically mixing.  Let $\phi$ be $d_\theta$-Lipschitz and denote by  $\mu$ the associated equilibrium state.  We suppose further that the family $\{U_N\}_{N}$ satisfies assumptions (\ref{ass1})-(\ref{ass5}).  Then 
\begin{equation*}\lim_{N\to\infty}\frac{r_\mu(U_N)}{\mu(U_N)}=	\begin{cases} 1 & \text{ if }z\text{ is non-periodic} \\
		1-e^{\phi^p(z)-pP(\phi)} & \text{if }z\text{ has prime period }p\end{cases}\end{equation*}\label{thm:mainpert} where $\phi^p(z)=\phi(z)+\phi(\sigma(z))+\cdots+\phi(\sigma^{p-1}(z))$.
	\end{prop}

\subsection{A symbolic model for the set $X$}

The map $T:X\to X$ is semi-conjugated to a full shift, let 
\begin{equation*}\begin{split}\Sigma=\{(\omega,\tau)=((\omega_i)_{i=0}^{\infty},(\tau_i)_{i=0}^{\infty})&\in \{0,1,\ldots,n-1\}^\mathbb{N}\times\{0,1,\ldots,m-1\}^\mathbb{N}  \\
	 &\,\,\,\,\,\,\,\,\,\,\,\,\,\,\,\,\,\,\,\,\,\,\,\,\,\,\,\,\,\,\,\,\,:\,(\omega_i,\tau_i)\in D\text{ for }i=0,1,\ldots\}.\end{split}\end{equation*}

With the shift map $\sigma:\Sigma\to\Sigma$ acting as the usual shift on each of the entries, i.e. $\sigma(\omega,\tau)=\sigma((\omega_i)_{i=0}^{\infty},(\tau_i)_{i=0}^{\infty})=((\omega_{i+1})_{i=0}^{\infty},(\tau_{i+1})_{i=0}^{\infty})$.  Denote by  $\Pi:\Sigma\to X$ the factor map defined by
	\begin{equation*} \Pi(\omega,\tau)=\lim_{k\to\infty} F_{(\omega_0,\tau_0)}F_{(\omega_1,\tau_1)}\cdots F_{(\omega_k,\tau_k)}(0,0).\end{equation*}  
		
For $l,k\in\mathbb{N}$ and $(\omega,\tau)\in\Sigma$ we define the $(l,k)$ cylinder set centred on $(\omega,\tau)$ to be 
		\begin{equation*} [(\omega,\tau)]_{(l,k)}=\{(\omega^\prime,\tau^\prime)\in\Sigma\,:\,\omega^{\prime }_i=\omega_i\text{ for }0\leq i < l\text{ and }\tau^{\prime}_j=\tau_j\text{ for }0\leq j < k\}.\end{equation*}	
			
We denote by $\mathcal{C}_k$ the set of all $(k,k)$-cylinders, that is

\begin{equation*}\mathcal{C}_k=\{[(\omega,\tau)]_{(k,k)}\,:\,(\omega,\tau)\in\Sigma\}.\end{equation*}	
	
The non-conformality of the map $T$ dictates that images of the cylinders $\mathcal{C}_k$ do not form an optimal cover of the set $X$.  We therefore introduce approximate squares.  Let $\eta=\frac{\log m}{\log n}$ then we set 
\begin{equation*}\mathcal{R}_k=\{[(\omega,\tau)]_{([\eta k],k)}\,:\,(\omega,\tau)\in\Sigma\}.\end{equation*}  The image $\Pi(R)$ of $R\in\mathcal{R}_k$ is a rectangle of side length $n^{-\eta k}\times m^{-k}\approx m^{-k}\times m^{-k}$ intersected with $X$ and for the purposes of studying dimension these are the correct objects to study.

It is therefore necessary to consider the dynamics of a factor of $\Sigma$.  Let $\tilde{\Sigma}=\pi(D)^\mathbb{N}$ and denote by $\tilde{\sigma}:\tilde{\Sigma}\to\tilde{\Sigma}$ the associated shift.  We let $\tilde{\pi}:\Sigma\to\tilde{\Sigma}$ denote the map $\tilde{\pi}(\omega,\tau)=\tau$.

For $\tau\in \tilde{\Sigma}$ we define the cylinder set of length $k$ around $\tau$ to be the set
\begin{equation*}[\tau]_k=\{\tau^\prime=(\tau^\prime_i)_{i=0}^{\infty}\,:\,\tau_i=\tau^\prime_i\text{ for all }i<k\}\end{equation*} and we denote the collection of such cylinders by
\begin{equation*}\tilde{\mathcal{C}}_k=\{[\tau]_k\,:\,\tau\in\tilde{\Sigma}\}.\end{equation*}

\section{Proof of Theorem \ref{thm:boxMark}}
\label{sec:boxMark}

	In this section we prove Theorem \ref{thm:boxMark}.  Suppose that $U\subset X$ is a Markov set and that $T:X_U\to X_U$ is topologically mixing.  The measure of maximal entropy $\mu_{{\rm max}}$ for $\sigma:\Sigma\to\Sigma$ corresponds to the equilibrium state associated to the potential $\phi=-\log\# D$, i.e. the evenly weighted Bernoulli measure. Similarly, the measure of maximal entropy $\tilde{\mu}_{{\rm max}}$ for $\tilde{\sigma}:\tilde{\Sigma}\to\tilde{\Sigma}$ corresponds to the equilibrium state associated to $\tilde{\phi}=-\log\#\pi(D)$.  Let $V\subset \Sigma$ denote a finite union of cylinders of length say $l$ for which $\Pi(V)=U$.
	
A easily verified property of the measures $\mu_{{\rm max}}$ and $\tilde{\mu}_{{\rm max}}$ is that if $A\subset\Sigma$ $B\subset\tilde{\Sigma}$ consist of a finite union of cylinder sets of length, say $l(A),l(B)$ then
	\begin{equation}\label{eq:countingbasic}\begin{split}\mu_{{\rm max}}(A)& =(\# D)^{-l(A)}\#\{C\in \mathcal{C}_{l(A)}\,:\,C\subset A\}\\
	\tilde{\mu}_{{\rm max}}(B)& =(\# \tilde{D})^{-l(B)}\#\{C\in \tilde{\mathcal{C}}_{l(B)}\,:\,C\subset B\}	.\end{split}\end{equation}  

Unpacking the definition of the escape rate of $\mu_{{\rm max}}$ through $U$ we see that for fixed $N$ and $\epsilon>0$ there exists $k_0$ such that 

	\begin{equation}\label{eq:countingstuff} e^{-k\epsilon} \leq \frac{\# \{C\in C_{k+l} \,:\, \sigma^i(C)\cap V=\emptyset \text{ for }0\leq i <k\}}{(\# D)^{k+l}e^{-k r_{\mu_{{\rm max}}}(V)} } \leq e^{k\epsilon} \end{equation} for all $k\geq k_0$.

	Due to the fact that the correct cover of the set $X_{U}$ comes not from images of cylinder sets $\mathcal{C}_k$ but from approximate squares $\mathcal{R}_k$ we are required to consider the dynamics of the system $\tilde{\sigma}:\tilde{\pi}\left(\Sigma\setminus\bigcup_{k=0}^{\infty}\sigma^{-k}(V)\right)\to \tilde{\pi}\left(\Sigma\setminus\bigcup_{k=0}^{\infty}\sigma^{-k}(V)\right)$, i.e. the projection of the (symbolic) survivor set.

	We observe that by setting $\tilde{V}=\{A\in\tilde{\mathcal{C}}_{l}\,:\,\tilde{\pi}^{-1}(A)\setminus V=\emptyset\}$ we have that 

	\begin{equation*}\tilde{\pi}\left(\Sigma\setminus\bigcup_{k=0}^{\infty}\sigma^{-k}(V)\right)=\tilde{\Sigma}\setminus\bigcup_{k=0}^{\infty}\tilde{\sigma}^{-k}(\tilde{V}),\end{equation*}i.e. the projection of the survivor set is the survivor set associated with the hole $\tilde{U}$.  Trivially we have that $\emptyset \subseteq \tilde{V} \subseteq \tilde{\pi}(V)$ and we remark that both of these extremes may be realised.
		
We now prove Theorem \ref{thm:boxMark}.

	\begin{proof}Fix $k$ and consider the following cover of $X_{U}$ 

	\begin{equation*}\begin{split}D_{U,k}=\Pi&\{ R\in\mathcal{R}_k \,:\,\sigma^i(R)\cap V=\emptyset \text{ for }0\leq i< [\eta k]-l \\ & \,\,\,\,\,\,\,\,\,\,\,\,\,\,\,\,\,\, \text{ and }\tilde{\sigma}^i(\tilde{\pi}(R))\not\in \tilde{V}\text{ for }[\eta k] \leq i < k-l\}.\end{split}\end{equation*}

	Clearly, the condition $\sigma^i(R)\cap V= \emptyset$ for $0\leq i <[\eta k]-l$ imposes conditions on the first $[\eta k]$ symbols, and to count the number of $R\in\mathcal{R}_k$ satisfying this we use the escape rate of $\mu_{{\rm max}}$ through $V$.  While the later condition conditions on the later $k-[\eta k]$ symbols and we use the escape rate of $\tilde{\mu}_{{\rm max}}$ through $\tilde{V}$. 

	From the definition of the escape rate, in conjunction with (\ref{eq:countingbasic}) we see that for any $\epsilon>0$ there exists $k_0$ such that

	\begin{equation*}\# D_{V,k} \leq e^{(k-2 l(V))\epsilon} e^{-([\eta k]-l(V)) r_{\mu_{{\rm max}}}(V)} \left(\# D\right)^{[\eta k]} e^{-(k-[\eta k]-l(V)) r_{\tilde{\mu}_{{\rm max}}}(\tilde{V})} \left(\# \tilde{D}\right)^{k-[\eta k]} \end{equation*}	for all $k\geq k_0$.  Thus, denoting by $N_k(X_{U})$ the minimum number of boxes of side length $m^{-k}$ required to cover $X_{U}$ we have that 

	\begin{equation*}\begin{split}\overline{{\rm dim}}_B(X_{U}) & =  \limsup_{k\to\infty} \frac{\log N_k(X_{U})}{\log m^k} \\
		& \leq \limsup_{k\to\infty} \frac{\log \# D_{U,k}}{\log m^k} \\
		& \leq {\rm dim}_B(X) - \frac{1}{\log m}\left(\eta r_{\mu_{{\rm max}}}(V) + (1-\eta)r_{\tilde{\mu}_{{\rm max}}}(\tilde{V})\right) +\frac{\epsilon}{\log m}.\end{split} \end{equation*} 

	Since the above holds for all $\epsilon>0$ we deduce that 

	\begin{equation}\label{eq:boxupperbound} \overline{{\rm dim}}_B(X_{U}) \leq {\rm dim}_B(X) - \frac{1}{\log m}\left(\eta r_{\mu_{{\rm max}}}(V) + (1-\eta)r_{\tilde{\mu}_{{\rm max}}}(\tilde{V})\right).\end{equation}

	We now estimate the lower box dimension.  For fixed $k$ and $U$ let
	\begin{equation*}\tilde{D}_{U,k}=\Pi\{ R\in\mathcal{R}_k \,:\,\text{ there exists }(\omega,\tau)\in R \text{ such that }\sigma^i(\omega,\tau)\not\in V\text{ for all }i\geq 0\}.\end{equation*}	

	We shall estimate the cardinality of $\tilde{D}_{U,k}$.  Clearly, for each $C\in\mathcal{C}_{[\eta k]}$ satisfying $\sigma^i(C)\cap V=\emptyset$ for $0\leq i<[\eta k]-l$ has the property that $\Pi(C)\cap X_U\neq\emptyset$.  Furthermore, for each of these $C$ we may apply Lemma \ref{thm:glue} to the cylinder $\tilde{\sigma}^{[\eta k ]-l}\tilde{\pi}(C) \in \tilde{C}_{l}$ to see that there exists a constant $c>0$ such that
	\begin{equation*}\tilde{\mu}_{{\rm max}}\{y\in \tilde{\sigma}^{[\eta k ]-l(V)}\pi(C)\,:\,\tilde{\sigma}^i(y)\not\in \tilde{V}\text{ for }0\leq i < k-[\eta k]\} \geq e^{-(k-[\eta k])r_{\tilde{\mu}_{{\rm max}}}(\tilde{V})}(c+o(1)).\end{equation*}

Combining these estimates we deduce that each $C\in\mathcal{C}_{[\eta k]}$ such that $\sigma^i(C)\cap V=\emptyset$ for $0\leq i < [\eta k]$ contains at least $e^{-(k-[\eta k])r_{\tilde{\mu}_{{\rm max}}(\tilde{V})}}(c+o(1)) \pi(D)^{k-[\eta k]}$ approximate squares of side length $m^{-k}$ which each hit $\tilde{D}_{U,k}$ and so for fixed $\epsilon>0$ there exists $k_0$ such that

\begin{equation*}\# \tilde{D}_{U,k} \geq e^{-\epsilon k } e^{-[\eta k]r_{\mu_{{\rm max}}}(V)} (\# D)^{[\eta k]} e^{-(k-[\eta k])r_{\tilde{\mu}_{{\rm max}}}(\tilde{V})}(c+o(1)) \pi(D)^{k-[\eta k]}.\end{equation*}

Any approximate square in $\tilde{D}_{U,k}$ necessarily contains a point $(\omega,\tau)$ for which $\sigma^i(\omega,\tau)\not\in V$ for all $i\geq 0$.  Thus, if $\mathcal{V}_{k}$ is an optimal cover of $X_U$ by boxes of side length $m^{-k}$ then each $\Pi(R)\in \tilde{D}_{U,k}$ necessarily intersects some element of $\mathcal{V}_{k}$.  On the other hand each $V\in\mathcal{V}_k$ intersects at most $9$ elements of $\tilde{D}_{U,k}$ and so 
	\begin{equation*}\# \tilde{D}_{U,k} \leq 9 \# \mathcal{V}_k=9 N_k(X_U).\end{equation*} 
	
	  Thus
	\begin{equation*} \underline{{\rm dim}}_B(X_{U}) = \liminf_{k\to\infty} \frac{\log N_k(X_U)}{\log m^{-k}} \geq  {\rm dim}_B(X) - \frac{1}{\log m}\left(\eta r_{\mu_{{\rm max}}}(V) + (1-\eta)r_{\tilde{\mu}_{{\rm max}}}(\tilde{V})\right) -\frac{\epsilon}{\log m}.\end{equation*}
	Letting $\epsilon\to 0$ and combining with equation (\ref{eq:boxupperbound}) we see that 
	\begin{equation}\label{eq:boxdim}{\rm dim}_B(X_U)={\rm dim}_B(X) - \frac{1}{\log m}\left(\eta r_{\mu_{{\rm max}}}(V) + (1-\eta)r_{\tilde{\mu}_{{\rm max}}}(\tilde{V})\right).\end{equation}  Finally, we observe that the map $\Pi:\Sigma\to X$ is one to one almost everywhere for the measure $\mu_{\max}$ and that this measure is projected to the measure of maximal entropy for $T:X\to X$.  It follows that the corresponding escape rates coincide which completes the proof.  

	\end{proof}

\section{Proof of Theorem \ref{thm:hausMark}}
\label{sec:hausMark}

\subsection{Upper bound}

The purpose of this section is to prove the following upper bound.

\begin{equation*}{\rm dim}_H(X_U)\leq \sup_{\underline{p}\in\Delta_D}\left\{{\rm dim}_H(\mu_{\underline{p}})-\frac{\eta r_{\mu_{\underline{p}}}(U)+(1-\eta)r_{\pi_*(\underline{p})}(\tilde{U})}{\log m}\right\}.\end{equation*}

To prove this we modify McMullen's original argument, this is the following two lemmas.

For $\tau\in\pi(D)^\mathbb{N}$ and $k\geq 1$ we let 

\begin{equation*}\underline{q}_k(\tau)=(q_{k,j}(\tau))_{j=0}^{m-1}=\left(k^{-1}\sum_{i=0}^{m-1}\chi_{[j]}(\tilde{\sigma}^i(\tau))\right)_{j=0}^{m-1}\end{equation*} and $\nu_{\tau,k}$ the associated Bernoulli measure.  We define a sequence of function $\Psi_k:\pi(D)^\mathbb{N}\to\mathbb{R}$ by
	
\begin{equation*}\Psi_k(\tau)=h_{\nu_{\tau,k}}(\tilde{\sigma})-r_{\nu_{\tau,k}}(\tilde{V})\end{equation*} where $\tilde{V}\subset\pi(D)^\mathbb{N}$ denotes the collection of cylinder sets associated with $\tilde{U}$,  
	
The analogue of McMullen's condition \cite{McM84}[Lemma 4] that we require is the following.

\begin{lemma}For any $\tau\in\pi(D)^\mathbb{N}$ we have
\begin{equation*}\liminf_{k\to\infty}\left( \Psi_{k-[\eta k]}(\tilde{\sigma}^{[\eta k]}(\tau))-\Psi_{[\eta k]}(\tau)\right) \leq 0.\end{equation*}	
	\label{thm:hausmcm1}\end{lemma}

\begin{proof}We first prove that 
	\begin{equation}\label{eq:mcmullen1}\liminf_{k\to\infty}\Psi_{k}(\tau)-\Psi_{[\eta k]}(\tau)\leq 0.\end{equation}  Suppose by way of contradiction that 
\begin{equation*}\liminf_{k\to\infty}\Psi_{k}(\tau)-\Psi_{[\eta k]}(\tau) > 0.\end{equation*}  Thus, there exists $\epsilon>0$ and a positive integer $k_0$ such that $\Psi_{k}(\tau)-\Psi_{[\eta k]}(\tau)\geq \epsilon$ for all $k\geq k_0$.  This implies that the sequence $\{\Psi_k(\tau)\}_k$ is unbounded, a contradiction.
	
For $\underline{q}\in\Delta_{\pi(D)}$ we write 
\begin{equation*} H(\underline{q})=-\sum_{j\in\pi(D)} q_j\log q_j, \,\,\,\,\,\,\,\,\,\,\,\,\,\,\,\,\,\,E(\underline{q})=P_{\tilde{\Sigma}_{\tilde{V}}}(\log q_{\cdot}).	\end{equation*}
	
We observe that $\Psi_k(\tau)=H(\underline{q}_k(\tau))+E(\underline{q}_k(\tau))$.  Furthermore, as both of these functions are concave we see that
\begin{equation*}\Psi_k(\tau)\geq\frac{[\eta k]}{k}\Psi_{[\eta k]}(\tau)+\frac{k-[\eta k]}{k}\Psi_{k-[\eta k]}(\tilde{\sigma}^{[\eta k]}\tau).\end{equation*}

Thus, combining this with (\ref{eq:mcmullen1}) we see that
\begin{equation*}\begin{split}\liminf_{k\to\infty} \left( \Psi_{k-[\eta k]}(\tilde{\sigma}^{[\eta k]}(\tau))-\Psi_{[\eta k]}(\tau)\right) & = \frac{1}{1-\eta} \liminf_{k\to\infty} \left( \frac{k-[\eta k]}{k}\Psi_{k-[\eta k]}(\tilde{\sigma}^{[\eta k]}(\tau))\right.\\
	&\,\,\,\,\,\,\,\,\,\,\,\,\,\,\,\,\,\,\,\,\,\,\,\,\,\,\,\,\,\,\,\,\,\,\,\,\,\,\,\,\,\,\,\,\,\,\,\,\,\,\,\,\,\,\,\,\,\,\,\,\left.+\frac{[\eta k]}{k}\Psi_{[\eta k]}(\tau)-\Psi_{[\eta k]}(\tau)\right) \\
	& \leq \frac{1}{1-\eta} \liminf_{k\to\infty} \left(\Psi_{k}(\tau)-\Psi_{[\eta k]}(\tau)\right) \leq 0.\end{split} \end{equation*} This completes the proof.
			\end{proof}

We now prove the upper bound for Theorem \ref{thm:hausMark}.

\begin{proof}Fix $\epsilon>0$ and choose $\gamma_1,\gamma_2,\ldots,\gamma_m\in [0,\log \#\pi(D)]$ such that $[0,\log \#\pi(D)]=\cup_{i=1}^m (\gamma_i-\epsilon,\gamma_i+\epsilon)$.  For $\underline{p}\in\Delta_D$ let $q(\underline{p})\in\Delta_{\pi(D)}$ be the probability vector defined by $q(\underline{p})_j=\sum_{i:(i,j)\in D}p_{i,j}$ and let
\begin{equation*}
\Gamma_i=\left\{\underline{p}\in \Delta_D: H(q(\underline{p}))+E(q(\underline{p}))>\gamma_i-2\epsilon\right\}.
\end{equation*}
For $\delta>0$, positive integers $k$ and $1\leq j \leq M$  we choose a positive integer $N(M)$ and $p_1,\ldots,p_{N(M)}\in\Gamma_i$ such that $\Gamma_i\subset \cup_{i=1}^{N(M)} B_{\delta}(p_i)$. We will also let $R=\max_{j=1}^{M}N(M)$. For $1\leq i\leq M(N)$ we let $\mathcal{R}_{k,i,j}\subseteq \mathcal{R}_k$ consist of those approximate squares satisfying
	\begin{enumerate}
		\item[i.] For $0\leq l < [\eta k]$ we have $\sigma^l(R)\cap V = \emptyset$.
		\item[ii.] For $[\eta k]  \leq l < k$ we have $\tilde{\sigma}^l(\tilde{\pi}(R))\cap \tilde{V} = \emptyset$.
		\item[iii.] We have $([\eta k]^{-1}\sum_{s=0}^{[\eta k]-1} \chi_{[d]}(\sigma^s(\omega,\tau)))_{d\in D}\in B_\delta(p_i)$.
		\item[iv.] We have $((k-[\eta k])^{-1}\sum_{s=[\eta k]}^{k-1} \chi_{\tilde{d}}(\tilde{\sigma}^s(\tau)))_{\tilde{d}\in \pi(D)}\in (\gamma_j-\epsilon,\gamma_j+\epsilon)$.
		\item[v.] We have that $\Psi_{k-[\eta k]}(\tilde{\sigma}^{[\eta k]}\tau) \leq \Psi_{[\eta k]}(\tau)+\epsilon$. \label{line:mcmullen}
	\end{enumerate}
	
We observe that by continuity of the map $\underline{p}\mapsto h_{\mu_{\underline{p}}}(\sigma)$ there exists $\delta$ small enough such that  

\begin{equation*}\# \mathcal{R}_{k,i,j} \leq e^{ [\eta k] ( h_{\mu_{\underline{p}_i}}(\sigma) - r_{\mu_{\underline{p}_i}}(V) ) + (k-[\eta k])(h_{\tilde{\pi}_*(\mu_{\underline{p}_i})}(\tilde{\sigma}) -  r_{\tilde{\pi}_*(\mu_{\underline{p}_i})}(\tilde{V}) )  + 2k\epsilon}\end{equation*}
	
By virtue of (\ref{line:mcmullen}) in combination with Lemma \ref{thm:hausmcm1} we see that for any $k_0$ we have

\begin{equation}\label{cover}X_U \subseteq \bigcup_{k\geq k_0}\bigcup_{j=1}^{M} \bigcup_{i=1}^{M(N)}  \bigcup_{R\in\mathcal{R}_{k,i,j}}\Pi(R).\end{equation}

We define
\begin{equation*}\alpha=\sup_{\underline{p}\in \Delta_D}\left\{\frac{\eta h_{\mu_{\underline{p}}}(\sigma)+(1-\eta) h_{\tilde{\pi}_*(\mu_{\underline{p}})}(\tilde{\sigma})-\eta r_{\mu_{\underline{p}}}(V)-(1-\eta)r_{\tilde{\pi}_*(\mu_{\underline{p}})}(\tilde{V})+3\epsilon}{\log m}\right\}\end{equation*}
and use (\ref{cover}) to calculate 

	\begin{equation}\begin{split}\mathcal{H}_{m^{-k_0}}^\alpha (X_U) & \leq \sum_{k\geq k_0}\sum_{j=1}^M\sum_{i=1}^{N(M)} m^{-k\alpha}\# \mathcal{R}_{k,i,j} \\
		& \leq \sum_{k\geq k_0}\sum_{j=1}^M\sum_{i=1}^{N(M)}m^{-k\alpha}e^{ [\eta k] ( h_{\mu_{\underline{p}_i}}(\sigma) - r_{\mu_{\underline{p}_i}}(V) ) + (k-[\eta k])(h_{\tilde{\pi}_*(\mu_{\underline{p}_i})}(\tilde{\sigma}) -  r_{\tilde{\pi}_*(\mu_{\underline{p}_i})}(\tilde{V}) )  + 2k\epsilon} \\
	\label{eq:hausestimate}		& \leq 
R \sum_{k=k_0}^{\infty} m^{-k\epsilon} = R(1-m^{-\epsilon})m^{-k_0\epsilon}.
\end{split}\end{equation}
Letting $k_0\to\infty$ shows that 
\begin{equation*}{\rm dim}_H(X_U)\leq \sup_{\underline{p}\in \Delta_D}\left\{\frac{\eta h_{\mu_{\underline{p}}}(\sigma)+(1-\eta) h_{\tilde{\pi}_*(\mu_{\underline{p}})}(\tilde{\sigma})-\eta r_{\mu_{\underline{p}}}(V)-(1-\eta)r_{\tilde{\pi}_*(\mu_{\underline{p}})}(\tilde{V})+3\epsilon}{\log m}\right\}\end{equation*}
letting $\epsilon\to 0$ completes the proof.
\end{proof}

\subsection{Lower bound}

We first give the proof in the case where $\tilde{U}=\emptyset$.  Let $l\in\mathbb{N}$ denote any integer for which $[i n^{-l},(i+1)n^{-l}]\times [j m^{-l},(j+1)m^{-l}]\cap U\neq \emptyset$ implies that $[i n^{-l},(i+1)n^{-l}]\times [j m^{-l},(j+1)m^{-l}]\subseteq U$.   

\begin{prop}Under the assumptions above we have
\begin{equation*}{\rm dim}_H(X_U)\geq \sup_{\underline{p}\in\Delta_D}\left\{{\rm dim}_H(\mu_{\underline{p}})+\frac{1}{\log n }\int \log\left(\frac{\mu_{\underline{p}}(\pi^{-1}(I_l(y))\cap U^c)}{\pi_*(\mu_{\underline{p}})(I_l(y))}\right) d\pi_*(\mu_{\underline{p}})(y)\right\}.\end{equation*}	\end{prop}
	
	\begin{proof}Let $\underline{p}\in \Delta_{D}$ and $\mu_{\underline{p}}$ denote the associated Bernoulli measure supported on $\Sigma$.  Let $V\subset \Sigma$ denote the symbolic representation of the set $U$.  Now fix $\epsilon>0$ and a positive integer $k$ and set 
		\begin{equation*}	\begin{split}\tilde{\mathcal{B}}_k  =\Big\{\tau=\tau_0 \tau_1 \cdots \tau_{k-1}\in \pi(D)^k&\,:\, \Big| k^{-1} \# \{1\leq i  \leq k\,:\,\tau_i  =j\}-\tilde{\pi}_*(\mu_{\underline{p}})[j]\Big|<\epsilon 	\\
					 &  \text{ for all }j\in\pi(D)\text{ and }[\tau_{k-l}\tau_{k-l+1}\cdots \tau_{k-1}]\cap\tilde{\pi}(V)=\emptyset\Big\}.\end{split}\end{equation*}
		
					We define a function $\rho_k:\tilde{\mathcal{B}}_k\times\tilde{\mathcal{B}}_k\to\{0,1\}$ by

					\begin{equation*}\rho(\tau,\tau^\prime)=\begin{cases} 1 & \text{ if }\tilde{\sigma}^{-i}(\tilde{\pi}(V))\cap [\tau\tau^\prime]\text{ for }k-l\leq i <k \\
						0 & \text{ otherwise }.\end{cases}\end{equation*}

Then for $\tau\in\tilde{\mathcal{B}}_k$ we set

\begin{equation*} F_k(\tau)=\#\{\tau^\prime\in\tilde{\mathcal{B}}_k\,:\,\rho(\tau,\tau^\prime)=1\}. \end{equation*}

Then we see that for $\tilde{\pi}_*(\mu_{\underline{p}})$-almost all $\tau=(\tau_i)_{i=0}^\infty\in\{0,1,\ldots,m-1\}^{\mathbb{N}}$ we have

\begin{equation}\label{eq:lbest1}\lim_{k\to\infty}\frac{1}{k}\log\left(F_k(\tau_0 \tau_1\cdots \tau_{k-1} ) \right)=h_{\tilde{\pi}_*(\mu_{\underline{p}})}(\sigma).\end{equation} 
		
For $\tau\in\tilde{\mathcal{B}}_k$ we let 
		
		\begin{equation*}\begin{split} P_k(\tau):=\#\big\{\omega=\omega_0 \omega_1\cdots \omega_{k-1}&\in\{0,1,\ldots,n-1\}^k\,:\,(\omega_i,\tau_i)\in D\text{ for }i=0,1,\ldots,k-1 \\
			&\,\,\,\,\,\,\,\,\,\,\,\,\,\,\,\,\,\,\,\,\,\,\,\,\, \text{ and } \sigma^{-i}(V)\cap [\omega,\tau] = \emptyset \text{ for }i=0,1,\ldots,k-l-1 \big\}.\end{split}\end{equation*} One may deduce that for any $\tau\in\tilde{\mathcal{B}}_k$ 

			\begin{equation*} P_k(\tau) \geq \prod_{i=0}^{k-1} z(\tau_i) \prod_{i=0}^{k-l-1} \frac{\mu_{\underline{p}}(\tilde{\pi}^{-1}[\tilde{\sigma}^i(\tau)]_l\cap V^c)}{\tilde{\pi}_*(\mu_{\underline{p}})[\tilde{\sigma}^i(\tau)]_l}. \end{equation*}		
		
An application of the ergodic theorem yields for $\tilde{\pi}_*(\mu_{\underline{p}})$-almost all $\tau\in\{0,1,\ldots,m-1\}^\mathbb{N}$

\begin{equation}\label{eq:lbest2}\liminf_{k\to\infty} \frac{1}{k}\log P_k(\tau_0 \tau_1\cdots \tau_{k-1}) \geq \int \log z(\tau_0) d\pi_*(\mu_{\underline{p}})(\tau) + \int \log\left(\frac{mu_{\underline{p}}(\tilde{\pi}^{-1}[\tau]_l\cap V^c)}{\pi_*(\mu_{\underline{p}})[\tau]_l}\right) d\tilde{\pi}_*(\mu)(\tau).\end{equation}
	
Next set

\begin{equation*}\begin{split}A_{k}(\tau)=\big\{\omega=\omega_0 \omega_1\cdots \omega_{k-1}&\in\{0,1,\ldots,n-1\}^k\,:\,(\omega_i,\tau_i)\in D\text{ for }i=0,1,\ldots,k-1 \\
	&\,\,\,\,\,\,\,\,\,\,\,\,\,\,\,\,\,\,\,\,\,\,\,\,\, \text{ and } \sigma^{-i}(V)\cap [\omega,\tau] = \emptyset \text{ for }i=0,1,\ldots,k-l-1 \big\}\end{split}\end{equation*}
and let 
\begin{equation*}\mathcal{B}_k=\bigcup_{\tau\in\tilde{\mathcal{B}}_k} A_{k}(\tau)\times\{\tau\}.\end{equation*}
We now let
\begin{equation*}\tilde{X}=\{(\tilde{\omega}_i,\tilde{\tau}_i)\in\mathcal{B}_k^{\mathbb{N}}\,:\,\rho_k(\tilde{\tau}_i,\tilde{\tau}_{i+1})=1\text{ for }i=0,1,\ldots\}\end{equation*}
and note that $\tilde{X}$ may be viewed as a subset of $\Sigma$ and in this case we see that   $\Pi(\tilde{X})\subset X_U$.  We will now construct a probability measure with support $\tilde{X}$.  We define $\nu$ to be the Markov measure defined as

\begin{equation*}\nu([\tilde{\omega}_0\tilde{\omega}_1\cdots\tilde{\omega}_r,\tilde{\tau}_0\tilde{\tau}_1\cdots\tilde{\tau}_r]) = (\#\tilde{\mathcal{B}}_k)^{-1}F_k(\tilde{\tau}_0)\prod_{i=0}^{r} (F_k(\tilde{\tau}_i) P_k(\tilde{\tau}_i))^{-1}\end{equation*}
	
Hence for an approximate square $[\tilde{\omega}_0\tilde{\omega}_1\cdots\tilde{\omega}_{[\eta r]},\tilde{\tau}_0\tilde{\tau}_1\cdots\tilde{\tau}_r]$ we have that

\begin{equation*}\nu([\tilde{\omega}_0\tilde{\omega}_1\cdots\tilde{\omega}_{[\eta r]},\tilde{\tau}_0\tilde{\tau}_1\cdots\tilde{\tau}_r]) = (\#\tilde{\mathcal{B}}_k)^{-1}F_k(\tilde{\tau}_0)\prod_{i=0}^{[\eta r]} (F_k(\tilde{\tau}_i) P_k(\tilde{\tau}_i))^{-1} \prod_{i=[\eta r]+1}^r F_k(\tilde{\tau}_i)^{-1}.\end{equation*}

Now combining equations (\ref{eq:lbest1}) and (\ref{eq:lbest2}) with the above yields for large enough $k$

\begin{equation*}\begin{split}\liminf_{r\to\infty}\frac{\nu([\tilde{\omega}_0\tilde{\omega}_1\cdots\tilde{\omega}_{[\eta r]},\tilde{\tau}_0\tilde{\tau}_1\cdots\tilde{\tau}_r])}{\log m^{-r}}& \geq \frac{1}{\log m}\Bigg(  h_{\tilde{\pi}_*(\mu)}(\sigma) + \eta \int \log z(\omega_0) d\tilde{\pi}_*(\mu)(\omega)\\
	& \,\,\,\,\,\,\,\,\,\,\,+\eta\log\left(\frac{\mu_{\underline{p}}(\tilde{\pi}^{-1}[\tau]_l\cap V^c)}{\tilde{\pi}_*(\mu_{\underline{p}})[\tau]_l}\right) d\pi_*(\mu_{\underline{p}})(\tau) \Bigg)-\epsilon\\
	& = \frac{1}{\log m} \Bigg(\eta h_{\mu_{\underline{p}}}(\sigma) + (1-\eta)h_{\tilde{\pi}_*(\underline{p})}(\tilde{\sigma})  \\
& \,\,\,\,\,\,\,\,\,\,\,	+ \eta\log\left(\frac{\mu_{\underline{p}}(\tilde{\pi}^{-1}[\tau]_l\cap V^c)}{\tilde{\pi}_*(\mu_{\underline{p}})[\tau]_l}\right) d\tilde{\pi}_*(\mu_{\underline{p}})(\tau) \Bigg)-\epsilon.\end{split}\end{equation*} Letting $\epsilon\to 0$ completes the proof.  
\end{proof}

We now deal with the other case, that is $\tilde{U}=\pi(U)$. We start with a couple of straightforward lemmas. 
\begin{lemma}
Let $\nu$ denote a $\sigma$-invariant ergodic probability measure and $\mu$ be the Bernoulli measure on $\Sigma$ where $\mu\circ\Pi^{-1}=\mu_{\dim}$. It follows that for $\nu$-almost all $(\omega,\tau)\in\Sigma$ we have
	\begin{equation*} \lim_{k\to\infty} \frac{\log \mu (R_k(\omega,\tau))}{\log m^{-k}} = {\rm dim}_H(X) \end{equation*} where $R_k(\omega,\tau)$ denotes the unique element $R\in \mathcal{R}_k$ containing $(\omega,\tau)$.\label{thm:mcmalmostsure}
\end{lemma}
\begin{proof}
Let $(\omega,\tau)\in\Sigma$, fix $k\in\N$ . We then have that by the definition of $\mu$
$$\log\mu(R_k(\omega,\tau))=-k\dim_H (X)\log m+\sum_{i=1}^{k}\log z(\omega_i)^{\eta}-\sum_{i=1}^{[\eta k]}\log z(\omega_i).$$
By the Birkhoff Ergodic Theorem for $\nu$ almost all $(\omega,\tau)\in\Sigma$
$$\lim_{k\to \infty}\frac{1}{k}\left(\sum_{i=1}^{k}\log z(\omega_i)^{\eta}-\sum_{i=1}^{[\eta k]}\log z(\omega_i)\right)=0$$
and thus 
$$\lim_{k\to\infty}\frac{\log\mu(R_k(\omega,\tau))}{-k\log m}=\dim_H X.$$
\end{proof}

The following result relates the typical symbolic local dimension to the real local dimension.
\begin{lemma}
Let $\nu$ denote an ergodic $\sigma$-invariant measure with support in $\Sigma$.  Set $\tilde{\nu}=\nu\circ\Pi^{-1}$ and suppose that for $\nu$-almost all $(\omega,\tau)\in\Sigma$ we have 
$\liminf_{k\to\infty}\frac{\log \nu(R_k(\omega,\tau))}{-k\log m}\geq s$
then
for $\tilde{\nu}$-almost all $x\in X$ we have
$$\liminf_{r\to 0}\frac{\log \tilde{\nu}(B(x,r))}{\log r}\geq s$$
and in particular 
$$\dim_H\tilde{\nu}\geq s.$$ 
\label{sym-loc}
\end{lemma}
\begin{proof}
We let $\epsilon>0$ and define a set $A(\epsilon)\subset X$ such that $x\in A$ if 
there exists $(\omega,\tau)\in\Sigma$ and $K(x)\in\N$ such that
\begin{enumerate}
\item[A.]
 $\Pi(\omega,\tau)=x$ and for all $k\geq K(x) $ we have that $\nu(R_k(\omega,\tau))\leq m^{-k(s-\epsilon)}$
\item[B.]
for all $k\geq K(x)$ we have that $B(x,m^{-{k+\epsilon}})\subset \Pi(R_k(\omega,\tau))$.
\end{enumerate}
We then have that if $x\in A(\epsilon)$ then for $k\geq K(y)$ and $m^{-{k(1+\epsilon)}}\geq r> m^{-{(k+1)(1+\epsilon)}}$ 
$$\tilde{\nu}(B(x,r))\leq m^{-k(s-\epsilon)}\leq m^{(-1-\epsilon)(s-\epsilon)}m^{\epsilon k(s-\epsilon)}r^{s-\epsilon}.$$
Thus if $x\in A(\epsilon)$ for all $\epsilon$ where $\epsilon^{-1}\in\N$ then
$$\liminf_{r\to 0}\frac{\log \tilde{\nu}(B(x,r))}{\log r}\geq s.$$

So we need to show that $\tilde{\nu}\left(\cup_{i=1}^{\infty}A(i^{-1})\right)=1$. This follows for condition A by the assumption in the lemma. Thus to have $\tilde{\nu}\left(\cup_{i=1}^{\infty}A(i^{-1})\right)\neq 1$ requires condition B. to fail on a set of positive measure. However if $\nu$ assigns positive measure to more than one row and more than one column then by the Birkhoff Ergodic Theorem condition B must be satisfied for any $\epsilon>0$ on a set of full measure. If $\nu$ is just supported on one row or one column then $\tilde{\nu}$ has one-dimensional support and is ergodic under either $x\to n x\mod 1$ or $x\to m x\mod 1$ in which case the lemma is well known and a simple exercise to prove.    
\end{proof}

We can now complete the lower bound in this case. 
\begin{prop}Under the assumptions above we have
\begin{equation*}{\rm dim}_H(X_U)\geq \sup_{\underline{p}\in\Delta_D}\left\{{\rm dim}_H(\mu_{\underline{p}})-\frac{r_{\mu_{\underline{p}}}(U)}{\log m}\right\}\end{equation*} \end{prop}
	\begin{proof}
We begin by observing that by the Gibbs property of $\mu_U$ there exists a constant $C>1$ such that for any cylinder set in the space $C_n(\omega,\tau)$ we have
		\begin{equation*} C^{-1} \leq \frac{\mu_U (C_k(\omega,\tau)) }{e^{\phi^k_U(\omega,\tau)-kP_{\Sigma_{U}}(\phi_U)}} \leq C \end{equation*}	for all $x\in \Sigma_U$ and $n\geq 1$.  Combining this with the analogous estimate for $\mu$ we see that there exists a constant $\tilde{C}>1$ such that for any $(\omega,\tau)\in \Sigma_U$ we have
		\begin{equation*} \tilde{C}^{-1} \leq \frac{\mu_U(C_k(\omega,\tau))}{\mu(C_k(x)) e^{-kP_{\Sigma_{U}(\phi_U)}}} \leq \tilde{C}.\end{equation*}

		Using this we see that for $\mu_U$ almost all $(\omega,\tau)\in\Sigma_{U}$ we have 

		\begin{equation*}\begin{split}\frac{\log \mu_U(R_k(\omega,\tau))}{\log m^{-k}} & = \frac{\log \sum_{C_k(\omega^\prime,\tau^\prime)\subset R_k(\omega,\tau)\,:\,(\omega^\prime,\tau^\prime)\in\Sigma_{U}} \mu_U(C_k(\omega^\prime,\tau^\prime))}{\log m^{-k}} \\
			& \geq \frac{\log \sum_{C_k(\omega^\prime,\tau^\prime)\subset R_k(\omega,\tau)\,:\,(\omega^\prime,\tau^\prime)\in\Sigma_{U}} \mu(C_k(\omega^\prime,\tau^\prime)) }{\log m^{-k}} + \frac{\log \tilde{C}e^{-kP_{\Sigma_{U}(\phi_U)}} }{\log m^{-k}} \\
		& \geq \frac{\log \sum_{C_k(\omega^\prime,\tau^\prime)\subset R_k(\omega,\tau)} \mu(C_k(\omega^\prime,\tau^\prime)) }{\log m^{-k}} + \frac{\log \tilde{C}e^{-kP_{\Sigma_{U}(\phi_U)}} }{\log m^{-k}} \\
		& = \frac{\log \mu(R_k(\omega,\tau))}{\log m^{-k}} + \frac{\log \tilde{C}e^{-kP_{\Sigma_{U}(\phi_U)}} }{\log m^{-k}}.\end{split}\end{equation*}

		Taking liminfs and invoking Lemma \ref{thm:mcmalmostsure} and Lemma \ref{sym-loc} completes the proof.
		\end{proof}	

\section{Approximation arguments}
\label{sec:approx}

An important step in the proof of Theorems \ref{thm:boxmet} and \ref{thm:hausmet} is an argument where the geometric hole $B_{\epsilon}(z)$ is approximated by Markov holes.

In the case that $T$ is conformal it is reasonably straightforward to show that any equilibrium state corresponding to a H\"older continuous potential is doubling, which in turn provides the opportunity to use the ``$\delta$-annular decay property'' (see \cite{Buc99} for further details) - giving a simple method for going between finite unions of Markov holes and metric balls.  

If $T$ is non-conformal then this method no longer works as the equilibrium state may no longer be doubling.  The following lemma sidesteps this issue, showing directly that for measures such as the measure of maximal entropy or the measure of maximal dimension an analogous approximation argument holds.  

\begin{prop}Let $\mu_{{\rm max}}$ denote the measure of maximal entropy for $T:X\to X$ and $\tilde{\mu}_{{\rm max}}$ the measure of maximal entropy for $S:\pi(X)\to\pi(X)$.  Then for any $\delta>0$ there exists non-increasing functions $N,N^\prime:(0,1)\to\mathbb{N}$ and families $\{U_{N(\epsilon)}\}_{\epsilon},\{V_{N^\prime(\epsilon)}\}_{\epsilon}$ satisfying (\ref{ass1})-(\ref{ass5}) such that  $\Pi(V_{N^\prime(\epsilon)})\subset B_{\epsilon}(z) \subset \Pi(U_{N(\epsilon)}) $ and 
	\begin{subequations}\begin{align}(1-\delta)\mu_{{\rm max}} \left(\Pi\left(U_{N(\epsilon)}\right)\right) &\leq \mu_{{\rm max}}(B_{\epsilon}(z)) \leq (1+\delta)\mu_{{\rm max}}\left(\Pi\left(V_{N^\prime(\epsilon)}\right)\right)\label{eq:mainineq}\\
		(1-\delta)\tilde{\mu}_{{\rm max}}  \left(\Pi\left(\tilde{U}_{N(\epsilon)}\right)\right) & \leq \tilde{\mu}_{{\rm max}}(\tilde{B}_{\epsilon}(z)) \leq (1+\delta)\tilde{\mu}_{{\rm max}}\left(\Pi\left(\tilde{V}_{N^\prime(\epsilon)}\right)\right)\label{eq:mainineqtilde}\end{align}
		\end{subequations} for all $\epsilon\in (0,1)$.  Here \begin{equation*}\Pi(\tilde{U}_{N(\epsilon)})=\{x\in\pi(\Pi(U_{N(\epsilon)})\,:\,\pi^{-1}\{x\}\cap X\subset\Pi(U_{N(\epsilon)}) \}\end{equation*} with the corresponding definitions for $\Pi(\tilde{V}_{N^\prime(\epsilon)})$ and $\tilde{B}_{\epsilon}$.\label{thm:annulardecaymaxent}\end{prop}

\begin{proof}		We only prove the existence of the outer approximation $\{U_{N(\epsilon)}\}_{\epsilon}$ as the proof of the inner approximation $\{V_{N^\prime(\epsilon)}\}_{\epsilon}$ is analogous.  We first deduce (\ref{eq:mainineq}).  For $k\geq 1$ and $\epsilon>0$ set

			\begin{equation*}\tilde{\mathcal{U}}_{k,\epsilon}=\{ C\in \mathcal{C}_k\,:\, \Pi(C)\cap \partial B_\epsilon(z)\neq\emptyset\}.	\end{equation*}

		We claim that there exists constants $c,s>0$ and $\gamma\in (0,1)$ such that

						\begin{equation}\mu_{{\rm max}}\left(\bigcup_{U\in\mathcal{U}_{k,\epsilon}} \Pi(U)\right) \leq c \epsilon^{-s}\gamma^k \mu_{\rm max}(B_\epsilon(z))\label{eq:keyineq}\end{equation} for all $\epsilon>0$ and $k\geq 1$.  Observe that 
						\begin{equation*} \begin{split}\bigcup_{U\in\mathcal{U}_{k,\epsilon}} \Pi(U) \subseteq & B_{\epsilon+\sqrt{2}m^{-k}}(z)\setminus B_{\epsilon-\sqrt{2}m^{-k}}(z) \\
							&\subseteq \underbrace{\pi^{-1}[\pi(z)+\epsilon-\sqrt{2}m^{-k},\pi(z)+\epsilon+\sqrt{2}m^{-k}]}_{:=A} \\
							& \cup \underbrace{\pi^{-1}[\pi(z)-\epsilon+\sqrt{2}m^{-k},\pi(z)+\epsilon-\sqrt{2}m^{-k}] \cap B_{\epsilon+\sqrt{2}m^{-k}}(z)\setminus B_{\epsilon-\sqrt{2}m^{-k}}(z)}_{:=B} \\
							& \cup \underbrace{\pi^{-1}[\pi(z)-\epsilon-\sqrt{2}m^{-k},\pi(z)-\epsilon+\sqrt{2}m^{-k}]}_{:=C}. \end{split}\end{equation*}

							\begin{figure}[h]
							\begin{center}
							\includegraphics[width=1\textwidth]{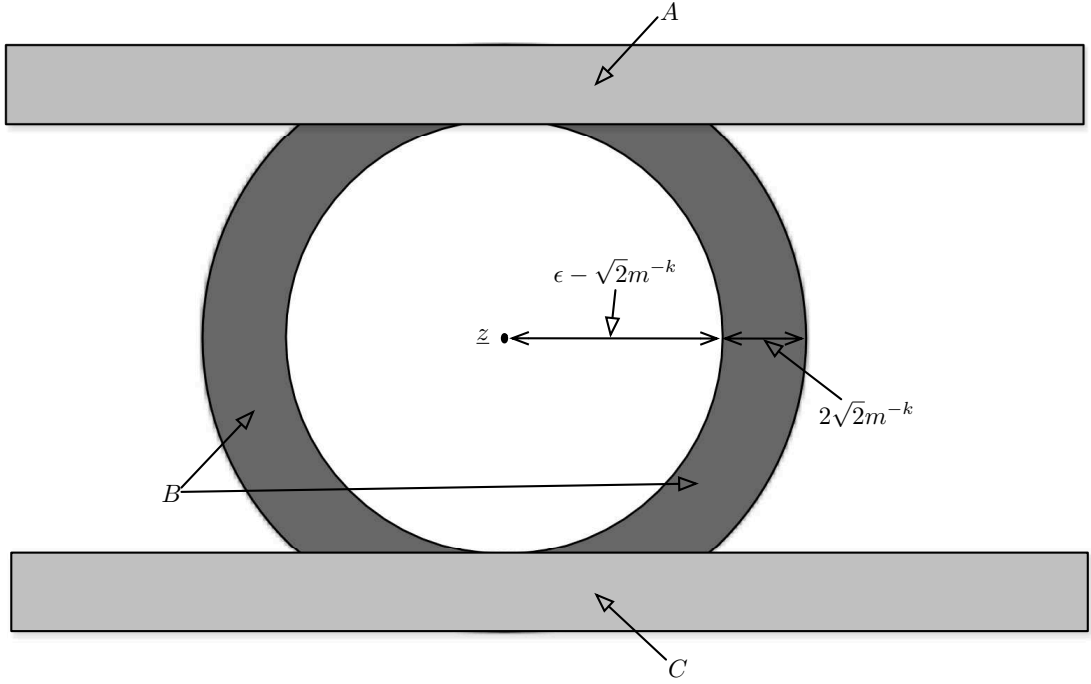}
							\end{center}
							\caption{The decomposition of the set $B_{\epsilon+\sqrt{2}m^{-k}}(z)\setminus B_{\epsilon-\sqrt{2}m^{-k}}(z)$ into the sets $A,B$ and $C$. }\label{fig:approx}
							\end{figure}	

						The measure $\pi_*(\mu_{\rm max})$ corresponds to the $(\# D^{-1}z(j))_{j=0}^{m-1}$-Bernoulli measure and so by a standard argument one may show the existence of constants $c_1,s_1>0$ such that $\pi_*(\mu_{\rm max})(U)\leq c_1({\rm diam}(U))^{s_1}$ for any open set $U\subset \mathbb{S}^1$.  Thus both (A) and (C) may be estimated as
						\begin{equation}\label{eq:AC} \mu_{\rm max}(A),\mu_{\rm max}(C) \leq c_1(2\sqrt{2}m^{-k})^{s_1} = c_2 m^{-ks_1}.\end{equation}

						We now estimate (B).  Decompose $\mu_{\rm max}=\pi_*(\mu_{\rm max})\times \mu_{y}$.  Suppose that $y\in\pi(X)$ has a unique base $m$-expansion and writing $y=\sum_{k=1}^{\infty} y_k m^{-k}$ and $x=\in \pi^{-1}\{y\} \cap X$ then it is easy to see that if $C\in\mathcal{C}_k$ is any cylinder with $(x,y)\in \Pi(C)$ then

						\begin{equation} \label{eq:measurefibre}\mu_y(\Pi(C))=\prod_{i=0}^{k-1} z(y_i)^{-1}\end{equation}

						Choose $N=N(\epsilon,k)$ so that $n^{-(N+1)} < 2\sqrt{2 \epsilon m^{-k}} \leq n^{-N}$.  We observe that for any interval $I\subset \pi^{-1}\{y\}$ of diameter $2\sqrt{2}m^{-k}$ there exist at most two $C_1,C_2\in\mathcal{C}_k$ such that $\Pi(C_1),\Pi(C_2)$ intersect $I$.  Accordingly, using (\ref{eq:measurefibre}) we may deduce that 
						\begin{equation}\label{eq:measurefibre2}\mu_y(I)\leq 2 \prod_{i=0}^{N-1} z(y_i)^{-1}.\end{equation}

						Next choose $L=L(\epsilon)$ so that $m^{-(L(\epsilon)+1)}<2\epsilon \leq m^{-L(\epsilon)}$.  Then the projection of the ball $\pi(B_{\epsilon-\sqrt{2}m^{-k}}(z))$ is covered by at most two $m$-adic intervals of length $m^{-L}$ which we denote $I_1,I_2$.  Thus we see that 
						\begin{equation}\begin{split} \mu_{\rm max}(B) & \leq 4\int_{I_1\cup I_2} \prod_{i=0}^{N-1} z(y_i)^{-1} d(\pi_* \mu_{\rm max})(y) \\
							& \leq 4 \left(\int_{I_1\cup I_2} \prod_{i=0}^{L-1}z(y_i)^{-1} d(\pi_*\mu_{\rm max})(y) \right)\left(\#\pi(D)/\#D\right)^{N-L} \\
							& \leq 4 \left(\#\pi(D)/\#D\right)^{N-L}:=4\alpha^{N-L}.\label{eq:measureB}
						\end{split}\end{equation}	
						From the definitions of $L$ and $N$ we deduce that
						\begin{equation*} N-L  \geq -\frac{\log 2^{3/2+\eta^{-1}}nm^{-k/2}\epsilon^{1/2+\eta^{-1}}}{\log n}:=c_3\log c_4m^{k s_2}\epsilon^{-s_3} 
						 \end{equation*}
						While combining the above estimate with (\ref{eq:measureB}) we deduce that 

						\begin{equation}\label{eq:measureB1} \mu_{\rm max}(B) \leq 4 \left(c_4 m^{k s_2}\epsilon^{-s_3}\right)^{c_3/\log\alpha}. \end{equation}

						We observe that the assumption that there exists $0\leq j <m$ with $z(j)>1$ forces $\alpha<1$, thus $m^{s_2c_3/\log \alpha}<1$ which implies that the right hand side of (\ref{eq:measureB1}) decays exponentially in $k$.  Combining this with (\ref{eq:AC}) proves claim (\ref{eq:keyineq}).

						Now fix $\delta>0$ and choose $k=k(\epsilon)$ such that $c\epsilon^{-s}\gamma^k<\delta < c\epsilon^{-s}\gamma^{k-1}$.  Set 
						\begin{equation*}U_{N(\epsilon)}=\left\{C\in\mathcal{C}_k\,:\,\Pi(C)\cap B_\epsilon(z)\neq\emptyset\right\}\}.\end{equation*}

						Then (\ref{eq:keyineq}) implies that
						\begin{equation*}\begin{split} \mu_{\rm max}(\Pi(U_{N(\epsilon)})) & \leq \mu_{\rm max}(B_\epsilon(z)) + \mu_{\rm max}\left(\bigcup_{U\in\mathcal{U}_{k,\epsilon} } \Pi(U) \right) \\
							& \leq (1+\delta)\mu_{\rm max}(B_\epsilon(z))\end{split}\end{equation*} which shows (\ref{eq:mainineq}).  We now show that the family $\{U_{N(\epsilon)}\}_{\epsilon}$ satisfies properties (\ref{ass1})-(\ref{ass5}).  Properties (\ref{ass1}), (\ref{ass2}) and (\ref{ass4}) follow immediately from the construction.  Finally, to see (\ref{ass5}) we observe that if $z$ is periodic with prime period $p$ and $C\in\mathcal{C}_p$ is the unique cylinder of length $p$ then for small enough $\epsilon$ we have that 
							\begin{equation*} \Pi(C)\cap T^{-p}(B_\epsilon(z)) \subseteq B_\epsilon(z).\end{equation*}  Since $\Pi(U_N)\subseteq B_{\epsilon+\sqrt{2}m^{-k(\epsilon)}}(z)$ we see that for small enough $\epsilon$ this property carries over to the family $\{U_{N(\epsilon)}\}_{\epsilon}$.
								
We now turn our attention to verifying (\ref{eq:mainineqtilde}) for the family $\{U_{N(\epsilon)}\}_{\epsilon}$.	We begin by observing that from the definition $\tilde{B}_\epsilon \subseteq \Pi(\tilde{U}_{N(\epsilon)})$.  Further, from the construction we have that 
\begin{equation*}\begin{split}\Pi(\tilde{U}_{N(\epsilon)}) \setminus\tilde{B}_\epsilon \subseteq [\pi(z)+\epsilon-\sqrt{2}m^{-N(\epsilon)},&\pi(z)+\epsilon+\sqrt{2}m^{-N(\epsilon)}] \\
	&\cup [\pi(z)-\epsilon-\sqrt{2}m^{-N(\epsilon)},\pi(z)-\epsilon+\sqrt{2}m^{-N(\epsilon)}]. \end{split}  \end{equation*} The equation (\ref{eq:AC}) and subsequent calculations verifies (\ref{eq:mainineqtilde}).
							\end{proof}

Using exactly the same arguments as in Proposition \ref{thm:annulardecaymaxent} we may deduce.

		\begin{prop}Let $\mu_{{\rm dim}}$ denote the measure of maximal dimension for $T:X\to X$.  Then for any $\delta>0$ there exists non-increasing functions $N,N^\prime:(0,1)\to\mathbb{N}$ and families $\{U_{N(\epsilon)}\}_{\epsilon},\{V_{N^\prime(\epsilon)}\}_{\epsilon}$ satisfying (\ref{ass1})-(\ref{ass5}) such that  $\Pi(V_{N^\prime(\epsilon)})\subset B_{\epsilon}(z) \subset \Pi(U_{N(\epsilon)}) $ and 
			\begin{equation*}\begin{split}(1-\delta)\mu_{{\rm dim}}\left(\Pi\left(U_{N(\epsilon)}\right)\right) & \leq \mu_{{\rm dim}}(B_{\epsilon}(z)) \leq (1+\delta)\mu_{{\rm dim}}\left(\Pi\left(V_{N^\prime(\epsilon)}\right)\right)\\
					(1-\delta)\pi_*(\mu_{\rm dim})  \left(\Pi\left(\tilde{U}_{N(\epsilon)}\right)\right) & \leq \pi_*(\mu_{\rm dim})(\tilde{B}_{\epsilon}(z)) \leq (1+\delta)\pi_*{\mu_{\rm dim}}\left(\Pi\left(\tilde{V}_{N^\prime(\epsilon)}\right)\right)\end{split}\end{equation*} for all $\epsilon\in (0,1)$.\label{thm:annulardecaymaxdim}\end{prop}
				
For $z\in X$ and $\epsilon>0$ small we are required to understand the dynamics of the projection of the survivor set $\pi(X_{B_\epsilon(z)})$ which is can be seen to be a survivor set in its own right.  Setting $\tilde{B}_{\epsilon}=\{y\in\pi(B_\epsilon(z))\,:\,\pi^{-1}\{y\}\cap X\subset B_\epsilon(z)\}$ we see that 
					\begin{equation*}\pi(X_\epsilon)=\pi(X)\setminus\bigcup_{k=0}^{\infty} S^{-k}(\tilde{B}_\epsilon).\end{equation*}

					The set $\tilde{B}_\epsilon$ may not necessarily be a ball, although clearly $\emptyset \subseteq \tilde{B}_\epsilon \subseteq \pi(B_\epsilon(z))$.  The next two Lemmas describe certain circumstances under which we may conclude that the set $\tilde{B}_\epsilon$ approximates either of these two extremes.

\begin{lemma}Let $\rho$ denote either 
\begin{enumerate}\item $\tilde{\mu}_{\rm max}$, the measure of maximal entropy for the system $S:\pi(X)\to\pi(X)$ or
	\item $\pi_*(\mu_{\rm dim})$, the projection of the measure of maximal dimension $\mu_{\rm dim}$ under the map $\pi$.\end{enumerate} 
	Let $\nu$ denote a $S$-invariant Borel probability measure with support in $X$ such that if $(0,0)\in X$ then we have
\begin{equation*}\underline{d}_\nu(0):=\liminf_{r\to 0}\frac{\log \nu(B_r(0))}{\log r}>0.\end{equation*}	
	
Then \begin{enumerate}
						\item If $0\leq z(i) \leq 1$ for all $i$ then \begin{equation*}\lim_{\epsilon\to 0}\frac{\rho(\tilde{B}_\epsilon)}{\rho(\pi(B_\epsilon(z)))} = 1\end{equation*} for $\nu$-almost all $z\in X$.

					\item Otherwise, $z(i)>1$ for some $i$ and for $\nu$-almost all $z\in X$ there exists $\epsilon_0=\epsilon_0(z)>0$ such that
					\begin{equation*}\tilde{B}_\epsilon= \emptyset\end{equation*} for all $0<\epsilon<\epsilon_0$.
					\end{enumerate}
					\label{thm:BorelCantelli1}
					\end{lemma}

\begin{proof}We only prove the case that $\rho=\tilde{\mu}_{\rm max}$ as the other case is proved in the same way.  We begin by observing that the assumption on $\nu$ implies that there exists constants $s,r_0>0$ such that
	\begin{equation*}\nu(B_r(0))\leq r^s\end{equation*} for all $0<r<r_0$.
	
Setting \begin{equation*}X_l:=\{z\in X\,:\,|S^l(\pi(z))|< (m n^{-1})^l\}.\end{equation*} Then by the invariance of $\nu$ we observe that $\nu(X_l)\leq (m n^{-1})^{l s}$ for large enough $l$.  Thus, by the Borel-Cantelli 
						lemma we deduce that 
						\begin{equation*}\nu\left(\bigcap_{k\geq 1}\bigcup_{ l\geq k} X_l\right)=0.\end{equation*} 

					Now let $z\in \bigcup_{k\geq 1} \bigcap_{l\geq k} X\setminus X_l$.  We first deal with the case where $0\leq z(i)\leq 1$ for all $i$.  Let $k=k(z)$ be chosen so that $|S^l(\pi(z))|\geq (m n^{-1})^l$ for all $l\geq k$. Choose $\epsilon_0=\epsilon_0(z)$ so that $0<\epsilon_0<n^{-k}$ then for any $0<\epsilon<\epsilon_0$	we set $n^{-(l(\epsilon)+1)}\leq \epsilon <n^{-l(\epsilon)}$.  From this we deduce that $|S^{l(\epsilon)}(\pi(z))| \geq (m n^{-1})^{l(\epsilon)}$ for all $0<\epsilon<\epsilon_0$ which in turn implies that there exists a unique cylinder $C_{l(\epsilon)}(z)\in\mathcal{C}_{l(\epsilon)}$ such that $z\in\Pi(C_{l(\epsilon)}(z))$.  Moreover, as $|S^{l(\epsilon)}(\pi(z))|\geq (m n^{-1})^{l(\epsilon)}$ we have that $\pi(B_\epsilon(z))\subseteq \pi(\Pi(C_{l(\epsilon)}(z)))$.  

					Moreover, the assumption that $0\leq z(i) \leq 1$ implies that $\pi^{-1}(\pi(C_{l(\epsilon)}(z)\cap X))\subseteq C_{l(\epsilon)}$.  We set 
					\begin{equation*}\tilde{C}_{l(\epsilon)}=\{ y \in \pi(B_\epsilon(z)\,:\,\pi^{-1}\{ y\}\subseteq B_{\epsilon}(z) \}.\end{equation*} 

					Clearly $\tilde{C}_{l(\epsilon)}\subseteq \tilde{B}_\epsilon \subseteq \pi(B_\epsilon(z))$ and so

					\begin{equation*} \frac{\tilde{\mu}_{\rm max}(\tilde{C}_{l(\epsilon)})}{\tilde{\mu}_{\rm max}(\pi(B_\epsilon(z)))}\leq \frac{\tilde{\mu}_{\rm max}(\tilde{B}_\epsilon)}{\tilde{\mu}_{\rm max}(\pi(B_\epsilon(z)))}\leq 1\end{equation*}

						\begin{figure}[h]
						\begin{center}
						\includegraphics[width=0.5\textwidth]{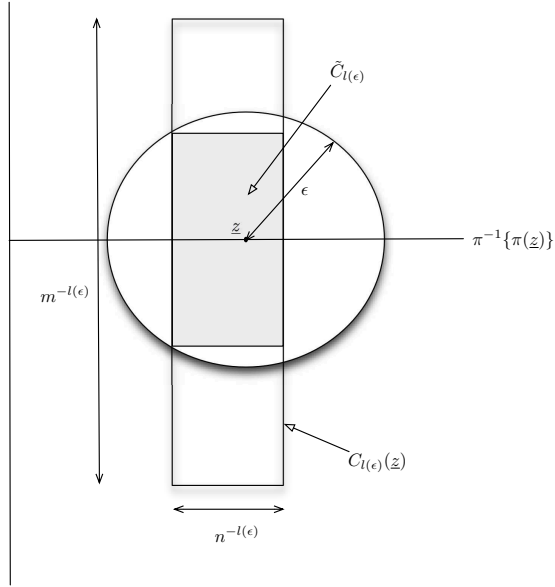}
						\end{center}
						\caption{Depiction of the set $\tilde{C}_{l(\epsilon)}$. }\label{fig:proj}
						\end{figure}

					It therefore suffices to show that $\frac{\tilde{\mu}_{\rm max}(\tilde{C}_{l(\epsilon)})}{\tilde{\mu}_{\rm max}(\pi(B_\epsilon(z)))}\to 1$ as $\epsilon\to 0$.  To see this we observe that  $\pi(B_\epsilon(z))\setminus \tilde{C}_{l(\epsilon)} \subset I_1 \cup I_2$ for two intervals of length at most $\sqrt{\epsilon^{2}-n^{-2l(\epsilon)}}\leq \frac{1}{2}m^{2\eta^{-1}}\epsilon^{2(1+\eta^{-1})}=c_2\epsilon^{2(1+\eta^{-1})} \epsilon^{2(1+\eta^{-1})}$.  Choose $r=r(\epsilon),\tilde{r}=\tilde{r}(\epsilon)$ so that
					\begin{equation*} m^{-r}  < \epsilon \leq m^{-(r-1)},\,\,\,\,\,\,\,\,\,\,\,\,\,\,\,\,\,\,   
						m^{-\tilde{r}}  < c_2\epsilon^{2(1+\eta^{-1})} \leq m^{-(\tilde{r}-1)}.
						\end{equation*}

					Then 

					\begin{equation*}\frac{\tilde{\mu}_{\rm max}(\tilde{C}_{l(\epsilon)})}{\tilde{\mu}_{\rm max}(\pi(B_\epsilon(z)))}  \geq 1-  4\frac{\# \pi(D)^{-\tilde{r}}}{\# \pi(D)^{-r}} \to 1 \end{equation*} as $\epsilon\to 0$.

					Now suppose that there exists $i$ such that $z(i)>1$, in which case for $\nu$-almost all $z\in X$ the set $\pi^{-1}\{\pi(z)\}\cap X$ contains infinitely many points.  Let $z\in\bigcup_{k\geq 1}\bigcap_{l\geq k} X\setminus X_l$ and suppose further that there exists $\tilde{z}\neq z$ such that $\pi(\tilde{z})=\pi(z)$.  Choose $0<\epsilon_0<2^{-1}|z-\tilde{z}|$.  For $0<\epsilon<\epsilon_0$ we let $k=k(\epsilon)$ be chosen so that $n^{-k} < \epsilon \leq n^{-(k-1)}$, in which case we see that there exist $C_1,C_2\in\mathcal{C}_k$ such that $z\in\Pi(C_1)$, $\tilde{z}\in\Pi(C_2)$, $\pi(B_\epsilon(z))\subseteq \pi(C_1)=\pi(C_2)$.  Thus, for any $y\in\pi(B_\epsilon(z))$ there exists $\tilde{\underline{y}}\in \Pi(C_2)$ such that $\pi(\tilde{\underline{y}})=y$, which implies that $\tilde{B}_\epsilon=\emptyset$.  This completes the proof.

							\begin{figure}[h]
							\begin{center}
							\includegraphics[width=0.75\textwidth]{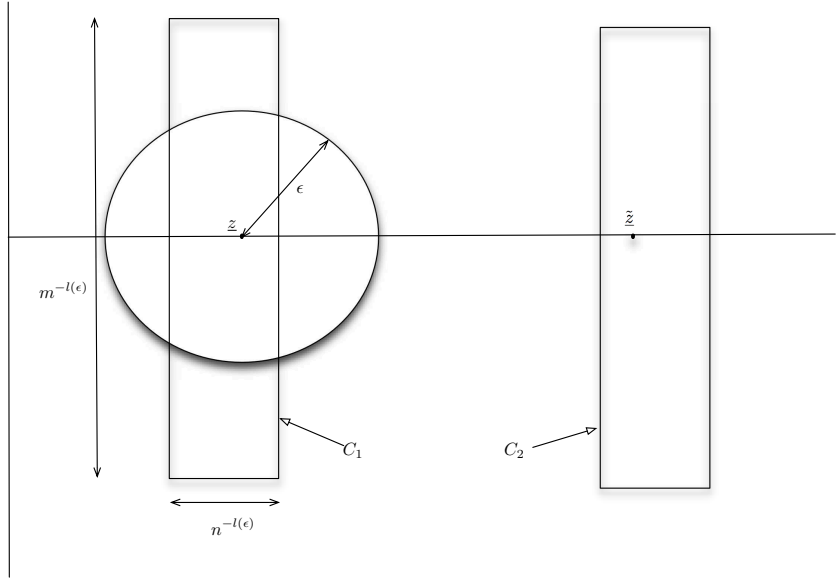}
							\end{center}
							\caption{Depiction of the sets $C_1$ and $C_2$. }\label{fig:proj2}
							\end{figure}

\end{proof}

\section{Metric holes}
\label{sec:metric}

In this section we prove Theorems \ref{thm:boxmet} and \ref{thm:hausmet}. Theorem \ref{thm:boxmet} is proved by a combination of using Theorem \ref{thm:boxMark} and the results from Section 5. Theorem \ref{thm:hausmet} is proved by combining Theorem \ref{thm:hausMark}, results from section 5 and a simple perturbation result.

\subsection*{Proof of Theorem \ref{thm:boxmet}} 
Fix a Borel probability measure $\nu$ supported on $X$ which satisfies the assumption of Lemma \ref{thm:annulardecaymaxdim}.  Let $z\in X$ be a generic point for $\nu$.  We first deal with the case that $0\leq z(j) \leq 1$ for all $0\leq j < m$.  Fix $\delta>0$ and for $\epsilon>0$ we let $\{U_{N(\epsilon)}\}_{\epsilon}$ and $\{V_{N^\prime(\epsilon)}\}_{\epsilon}$ denote the associated outer and inner approximations guaranteed by Proposition \ref{thm:annulardecaymaxent}.
	
Applying Theorem \ref{thm:boxMark} we see that

\begin{equation}\begin{split}\overline{{\rm dim}}_B(X_{B_\epsilon(z)}) & \geq {\rm dim}_B(X_{U_{N(\epsilon)}}) \\
	& = {\rm dim}_B(X)-\frac{1}{\log n}r_{\mu_{{\rm max}}}(U_{N(\epsilon)})-\left(\frac{1}{\log m}-\frac{1}{\log n}\right)r_{\tilde{\mu}_{{\rm max}}}(\tilde{U}_{N(\epsilon)}).\label{eq:boxmet1}\end{split}\end{equation}	
	
Applying Proposition \ref{thm:mainpert} to the above we see that

\begin{equation}\begin{split}\overline{{\rm dim}}_B(X_{B_\epsilon(z)}) & = {\rm dim}_B(X)-\frac{d_B(z)}{\log n}\mu_{\max}(U_{N(\epsilon)})-\tilde{d}_B(z)\left(\frac{1}{\log m}-\frac{1}{\log n}\right)\tilde{\mu}_{{\rm max}}(\tilde{U}_{N(\epsilon)})\} \\
	&\,\,\,\,\,\,\,\,\,\,\,\,\,\,\,\,\,\,\,\,\,\,\,\,\,\,\,\,\,\,\,\,\,\,\,\,\,\,\,\,\,\,\,\,\,\,\,\,+o(\max\{\mu_{\max}(U_{N(\epsilon)}),\tilde{\mu}_{{\rm max}}(\tilde{U}_{N(\epsilon)})\}),\label{eq:boxmet2}\end{split}\end{equation} where $d_B(z)$ and $\tilde{d}_B(z)$ are the factors that appear in Proposition \ref{thm:mainpert} which depend on the periodicity of $z$ and the potential associated to the measures $\mu_{\rm max}$ and $\tilde{\mu}_{\rm max}$ (which are nothing more than $-\log$ of the entropy in both cases).

Applying Proposition \ref{thm:annulardecaymaxent} we see that

\begin{equation}\begin{split}\overline{{\rm dim}}_B(X_{B_\epsilon(z)}) & \geq {\rm dim}_B(X)-(1-\delta)^{-1}\left(\frac{d_B(z)}{\log n}\mu_{\max}(B_\epsilon(z))+\tilde{d}_B(z)\left(\frac{1}{\log m}-\frac{1}{\log n}\right)\tilde{\mu}_{{\rm max}}(\tilde{B}_\epsilon(z))\right) \\
	&\,\,\,\,\,\,\,\,\,\,\,\,\,\,\,\,\,\,\,\,\,\,\,\,\,\,\,\,\,\,\,\,\,\,\,\,\,\,\,\,\,\,\,\,\,\,\,\,\,\,\,\,\,\,\,\,\,\,\,\,\,\,\,\,\,\,\,\,\,\,\,\,\,\,\,\,\,\,\,\,\,\,\,\,\,\,\,\,\,\,\,\,\,\,\,\,\,\,\,\,\,\,\,\,\,\,\,\,\,\,+o(\max\{\mu_{\max}(B_\epsilon(z)),\tilde{\mu}_{{\rm max}}(\tilde{B}_\epsilon(z)\}) \\
	& \geq {\rm dim}_B(X)-(1-\delta)^{-1}\left(\frac{d_B(z)}{\log n}\mu_{\max}(B_\epsilon(z))+\tilde{d}_B(z)\left(\frac{1}{\log m}-\frac{1}{\log n}\right)\tilde{\mu}_{{\rm max}}(\pi(B_\epsilon(z)))\right) \\
		&\,\,\,\,\,\,\,\,\,\,\,\,\,\,\,\,\,\,\,\,\,\,\,\,\,\,\,\,\,\,\,\,\,\,\,\,\,\,\,\,\,\,\,\,\,\,\,\,\,\,\,\,\,\,\,\,\,\,\,\,\,\,\,\,\,\,\,\,\,\,\,\,\,\,\,\,\,\,\,\,\,\,\,\,\,\,\,\,\,\,\,\,\,\,\,\,\,\,\,\,\,\,\,\,\,\,\,\,\,\,+o(\max\{\mu_{\max}(B_\epsilon(z)),\tilde{\mu}_{{\rm max}}(\pi(B_\epsilon(z))\}) \end{split}. \label{eq:boxmet3}\end{equation}

Finally, we observe the assumption that $0\leq z(j) \leq 1$ for all $j$ implies that $\pi_*(\mu_{\max})=\tilde{\mu}_{\max}$ and so from (\ref{eq:boxmet3}) we see that
\begin{equation*}\overline{{\rm dim}}_B(X_{B_\epsilon(z)})\geq{\rm dim}_B(X) - (1-\delta)^{-1} \frac{1}{\log m}\left(\eta d_B(z)+(1-\eta)\tilde{d}_B(z)\right)\mu_{\rm max}(B_\epsilon(z)) + o(\mu_{\rm max}(B_\epsilon(z))) \end{equation*} which demonstrates the lower bound for the upper box dimension.  The proof of the remaining cases are similar and therefore omitted.\qed

\subsection*{Proof of Theorem \ref{thm:hausmet}}  
Before proving Theorem \ref{thm:hausmet} we first prove a technical lemma that is required in the proof of the upper bound.  For a Markov set $U$ we let \begin{equation*}F_U:=\sup_{\underline{p}\in\Delta_D}\left\{{\rm dim}_H(\mu_{\underline{p}})-\frac{\eta \mu_{\underline{p}}(U)+(1-\eta)\pi_*(\mu_{\underline{p}})(U)}{\log m}\right\}\end{equation*}

\begin{lemma}Let $\mu_{{\rm dim}}$ denote the measure of maximal dimension for $T:X\to X$. Suppose that $\{U_N\}_N$ is a nested family of Markov holes for which the symbolic analogues satisfy the hypotheses of Proposition \ref{thm:mainpert} then

\begin{equation*} F_{U_N}={\rm dim}_H(X)-\frac{\eta \mu_{dim}(U_N)+(1-\eta)\pi_*(\mu_{dim})(\tilde{U}_N)}{\log m}+o(\mu_{dim}(U_N)).\end{equation*}

	\label{thm:haussup}\end{lemma}

\begin{proof}
This is a classical problem in perturbation theory: given a smooth function $G_0$ whose properties we know, approximate the supremum of function $G=G_0+G_1$, where $G_1$ is small.

In our situation, the domain is the simplex $\Delta_D$ of Bernoulli measures and

\[
G_0(\underline{p}) = \dim_H(\mu_{\underline{p}}).
\]
We will not need the exact formula for $G_0$, we will only need the following properties:
\begin{itemize}
\item[i)] $\sup_{\underline{p}} G_0(\underline{p}) = \dim_H(X)$,
\item[ii)] this supremum is achieved at the McMullen distribution $p^{\rm dim}$ which is an interior point of $\Delta_D$. We will denote
\[
p^+ = \sup_{i,j} p_{i,j}^{\rm dim},\ \ p^- = \inf_{i,j} p_{i,j}^{\rm dim},
\] \item[iii)] this supremum is unique and the second derivative form $D^2 G_0$ at $p^{\rm dim}$ is negative definite. There exists $d>0$ such that
\[
G_0(p^{\rm dim}) - G_0(\underline{p}) \geq d (p^{\rm dim}-\underline{p})^2.
\]
\end{itemize}

We set

\[
G_1^{(N)}(\underline{p}) = - \frac \eta {\log m} \mu_{\underline{p}}(U_N) - \frac {1-\eta} {\log m} \mu_{\underline{p}}(\pi^{-1}(\tilde{U}_N)).
\]
Let us denote by $p^N$ the point at which $G^N = G_0 + G_1^N$ achieves its maximum. We are trying to estimate

\[
F_{U_N} = G^N(p^N).
\]

$U_N$ and $\pi^{-1}(\tilde{U}_N) \subset U_N$ are both finite unions of cylinders of level $l(N)$ and are both contained in some cylinder of level $\kappa l(N)$. The latter property implies that

\[
G_1^N(p^{\rm dim}) \geq -\frac 1 {\log m} (p^+)^{\kappa l(N)}.
\]
As $G_1^N$ is nonpositive, it implies that

\[
|p^{\rm dim} - p^N| \leq r_N= (d \log m)^{-1/2} (p^+)^{\kappa l(N)/2}
\]
for $N$ big enough. Also, for $N$ big enough,

\[
r_N < \frac 1 2 p^-.
\]
We restrict our attention to big $N$ and to $\underline{p} \in B_{r_N}(p^{\rm dim})$ from now on.

Let us consider now the former property. For any cylinder $C$ of level $l(N)$ and any Bernoulli measure $\mu_{\underline{p}}$, $\mu_{\underline{p}}(C)$ is a monomial of degree $l(N)$ in variables $p_{i,j}$:

\[
\mu_{\underline{p}}(C) = \prod_{i,j} p_{i,j}^{n_{i,j}},\ \ \sum_{i,j} n_{i,j} = l(N).
\]
Hence, for any $e\in T\Delta_D$ we have

\[
|D_e \mu_{\underline{p}}(C)| = |\sum_{i,j} e_{i,j} \frac {n_{i,j}} {p_{i,j}} \mu_{\underline{p}}(C)| \leq \frac {2l(N)} {p^-} \mu_{\underline{p}}(C)
\]
and

\[
|D_e G_1^N(\underline{p})| \leq \frac {2l(N)} {p^-} G_1^N(\underline{p}).
\]

It implies that for any $\underline{p}, \underline{p}' \in B_{r_N}(p^{\rm dim})$ we have

\[
\left|\frac {G_1^N(\underline{p}) - G_1^N(\underline{p}')} {G_1^N(\underline{p})}\right| \leq e^{2r_N l(N)/p^-}-1.
\]

As both $r_N$ and $G_1^N(p^{\rm dim})$ are exponentially small as function of $l(N)$, we can find $v>0$ not depending on $N$ such that for all $\underline{p}\in B_{r_N}(p^{\rm dim})$

\[
|G_1^N(p^{\rm dim}) - G_1^N(\underline{p})| \leq \left(G_1^N(p^{\rm dim})\right)^{1+v}.
\]

Thus, we have

\begin{equation*}\begin{split}
G_0(p^{\rm dim}) + G_1^N(p^{\rm dim}) \leq F_{U_N} & = G_0(p^N) + G_1^N(p^N) \\
& \leq G_0(p^{\rm dim}) + G_1^N(p^N) \\
& \leq G_0(p^{\rm dim}) + G_1^N(p^{\rm dim}) + O((\mu_{\rm dim}(U_N))^{1+v})
\end{split}\end{equation*}
and the proof is complete 

\end{proof}

We now deduce Theorem \ref{thm:hausmet}. 

\begin{proof}Fix a Borel probability measure $\nu$ supported on $X$ such that $\underline{{\rm dim}}_{\rm loc}(\pi_*(\nu))(0)>0$.  Let $z\in X$ be a generic point for $\nu$, further we assume that $z$ is non-periodic.  We will only deal with the case that $0\leq z(j) \leq 1$ for all $0\leq j < m$ as the proof of the other case is similar.  For the lower bound we let $\{U_{N(\epsilon)}\}_\epsilon$ denote the outer approximations guaranteed by Proposition \ref{thm:annulardecaymaxent}.
	
Applying Theorem \ref{thm:hausMark} we see that

\begin{equation}\begin{split}{\rm dim}_H(X_{B_\epsilon(z)}) & \geq {\rm dim}_H(X_{U_{N(\epsilon)}}) \\
	& \geq {\rm dim}_H(X)-\frac{r_{\mu_{\rm dim}}(U_{N(\epsilon)})}{\log m}.\label{eq:hausmet1}\end{split}\end{equation} 

Applying Proposition \ref{thm:mainpert} to the above we see that

\begin{equation}{\rm dim}_H(X_{B_\epsilon(z)})  = {\rm dim}_H(X)-\frac{\mu_{\rm dim}(U_{N(\epsilon)})}{\log m}+o(\mu_{ \rm dim}(U_{N(\epsilon)})).\label{eq:hausmet2}\end{equation} 

Applying Proposition \ref{thm:annulardecaymaxent} we see that

\begin{equation*}(\ref{eq:hausmet2})  \geq {\rm dim}_H(X)-(1-\delta)^{-1}\frac{1}{\log m}\mu_{\rm dim}(B_\epsilon(z))+o(\mu_{\rm dim}(B_\epsilon(z))) \end{equation*} which proves the lower bound.
	
	For the upper bound we take the inner approximations $\{V_{N^\prime(\epsilon)}\}_{\epsilon}$ guaranteed by Proposition \ref{thm:annulardecaymaxent}.\end{proof}
	
\bibliographystyle{abbrv}
\bibliography{dim_holes}

\end{document}